\documentclass[12pt]{article}
\usepackage{latexsym, amsmath, amsfonts, amsthm, amssymb}
\usepackage{times}
\usepackage{a4wide}
\usepackage{tikz,pgfplots}
\usepackage{mathrsfs}
\usetikzlibrary{arrows}

\def \btwn {\text{\rm btwn}}

\def \area {\text{\rm Area}}
\def \length {\text{\rm Length}}
\def \perimeter {\text{\rm Perimeter}}
\def \subgraph {\text{\rm Subgraph}}
\def \Proj {\text{\rm Proj}}

\def \RR {\mathbb R}

\def \ZZ {\mathbb Z}

\def \eps {\varepsilon}
\def \vphi {\varphi}

\def \cF {\mathcal F}

\def \cS {\mathcal S}

\def \cN {\mathcal N}

\newtheorem{theorem}{Theorem}[section]
\newtheorem{definition}[theorem]{Definition}

\newtheorem{lemma}[theorem]{Lemma}

\newtheorem{proposition}[theorem]{Proposition}
\newtheorem{corollary}[theorem]{Corollary}
\newtheorem{remark}[theorem]{Remark}

\def\myffrac#1#2 in #3{\raise 2.6pt\hbox{$#3 #1$}\mkern-1.5mu\raise 0.8pt\hbox{$
		#3/$}\mkern-1.1mu\lower 1.5pt\hbox{$#3 #2$}}

\def\qed{\hfill $\vcenter{\hrule height .3mm
		\hbox {\vrule width .3mm height 2.1mm \kern 2mm \vrule width .3mm
			height 2.1mm} \hrule height .3mm}$ \bigskip}

\def \i {{ \iota }}

\def \cG {\mathcal G}

\begin{document}

\title{Rigidity of Riemannian embeddings \\ of discrete metric spaces}
\author{Matan Eilat and Bo'az Klartag}
\date{}
\maketitle

\begin{abstract} Let $M$ be a complete, connected Riemannian surface and suppose that $\cS \subset M$ is
a  discrete subset. What can we learn about $M$ from the knowledge of all Riemannian distances  between pairs of points
of $\cS$? We prove that if the distances in $\cS$ correspond to the distances in a $2$-dimensional lattice, or more generally in an arbitrary net in $\RR^2$,
then $M$ is  isometric to the Euclidean plane.
We thus find that Riemannian embeddings
of certain discrete metric spaces are rather rigid. A corollary is
that a subset of $\ZZ^3$ that strictly contains $\ZZ^2 \times \{ 0 \}$
cannot be isometrically embedded in any complete Riemannian surface.
\end{abstract}

\section{Introduction}

The collection of distances between pairs of points in a fine net in a Riemannian manifold $M$ provides information
on the geometry of the underlying manifold.
A common theme in the mathematical literature is that
the geometric information on $M$ that one extracts from a discrete net is {\it approximate}.
As the net gets finer, it better approximates the manifold.
Unless one makes substantial assumptions about the manifold $M$, knowledge of all distances in the net
typically implies that various geometric parameters of $M$ can be estimated
to a certain accuracy.

\medskip The question that we address in this paper is slightly different: Is it possible to obtain {\it exact} geometric information on the manifold $M$ from knowledge of
the distances between pairs of points in a discrete subset of $M$?
We show that the answer
is sometimes affirmative.

\medskip Recall that a discrete set $L \subseteq \RR^n$ is a net if there exists $\delta > 0$ such that $d(x,L) < \delta$
for any $x \in \RR^n$. Here, $d(x, L) = \inf_{y \in L} | x -y|$ and $|x| = \sqrt{\sum_i x_i^2}$
for $x \in \RR^n$. For example, any $n$-dimensional lattice in $\RR^n$ is a net.
 We say that $L$ embeds isometrically in a Riemannian manifold $M$ if there exists $\iota: L \rightarrow M$
such that for all $x,y \in L$,
$$ d(\i(x),\i(y)) = |x-y|, $$
where $d$ is the Riemannian distance function in $M$. We prove the following:

\begin{theorem}
Let $M$ be a complete, connected, $2$-dimensional Riemannian manifold. Suppose that there exists
a net in $\RR^2$ that embeds isometrically in $M$.
Then the manifold $M$ is flat and it is isometric to the Euclidean plane.
\label{thm_12170}
\end{theorem}

The conclusion of Theorem \ref{thm_12170} does not hold if we merely assume that $M$ is a Finsler manifold rather than a Riemannian manifold. Indeed, we may modify the Euclidean metric on $\RR^2$ in a disc that is disjoint from the net $L \subseteq \RR^2$,
and obtain a Finsler metric that induces the same distances among points in the complement of the disc. This was proven
by  Burago and Ivanov \cite{BI2}.
Theorem \ref{thm_12170} allows us to conclude that certain discrete metric spaces
embed in $3$-dimensional Riemannian manifolds but not in $2$-dimensional ones:

\begin{corollary} Let $X \subseteq \RR^3$ be a discrete set
that is not contained in any
affine plane, yet there exists an affine plane $H \subset \RR^3$ such that $X \cap H$ is a net in $H$. Endow $X$ with the Euclidean metric. Then
$X$ does not embed isometrically in any $2$-dimensional, complete Riemannian manifold.
\label{cor_344}
\end{corollary}

In view of Corollary \ref{cor_344} we define the {\it asymptotic Riemannian dimension}
 of a metric space as the minimal
dimension of a complete Riemannian manifold in which it embeds isometrically.
(It is undefined if there is no such Riemannian manifold).
For example, Corollary \ref{cor_344} tells
us that the asymptotic Riemannian dimension of the metric space
$$ X = \left( \ZZ^2 \times \{ 0 \} \right) \cup \{ (0,0,1) \} \subseteq \RR^3 $$
is exactly $3$. It seems to us that the asymptotic Riemannian dimension captures the large-scale geometry of the metric space,
hence the word asymptotic. In contrast, in the case of a finite  metric space, any reasonable definition of Riemannian dimension
should impose topological constraints on the manifold, since any finite, non-branching metric space may be isometrically embedded
in a two-dimensional surface of a sufficiently high genus.
We are not yet sure whether the $n$-dimensional analog of Theorem \ref{thm_12170} holds true.
The following result is valid in any dimension:

\begin{theorem}
Let  $M$ be a complete, connected, $n$-dimensional Riemannian manifold.
Suppose that there exists
a net in $\RR^n$ that embeds isometrically in $M$.
Then $M$ is diffeomorphic to $\RR^n$.
\label{thm_1217_0}
\end{theorem}

In the case where the curvature tensor of $M$ from Theorem \ref{thm_1217_0} is assumed compactly supported, it is not too
difficult to prove that $M$ is isometric to the Euclidean space $\RR^n$, by reducing matters to solved partial cases
of the {\it boundary distance conjecture} of Michel \cite{M}.
This conjecture  suggests that in a simple Riemannian manifold with boundary,
the collection of distances between boundary points determines the Riemannian structure, up to an isometry.  Michel's conjecture has been
proven
in two dimensions by Pestov and Uhlmann \cite{PU}.

\begin{definition} We say that a subset $X$ of an $n$-dimensional, complete, connected Riemannian manifold $M$ is metrically rigid, if whenever $X$
 isometrically embeds in a complete, connected, $n$-dimensional Riemannian manifold $\tilde{M}$, necessarily $\tilde{M}$ is isometric to $M$.
 \end{definition}

Nets in the Euclidean plane are metrically rigid, and so are random instances of a Poisson process with uniform intensity in the plane,
as we argue below.
One interesting question
in this direction is the metric rigidity of discrete subsets  in complete, simply-connected
Riemannian manifolds of non-positive curvature. Another natural question is whether there exist finitary versions of Theorem \ref{thm_12170},
 in which we isometrically embed a large, finite chunk of the net $L$ and wish to obtain some geometric corollaries.

\medskip We proceed to describe a notion slightly more inclusive than that of a net, which also covers  instances
of Poisson processes. We call an open set $S \subseteq \RR^n$  a {\it sector} if there exist $x_0 \in \RR^n$ and an open,
connected set $U \subseteq S^{n-1} = \{ x \in \RR^n \, ; \, |x| = 1 \}$ such that $ S = \left \{ x_0 + r \theta \, ; \, \theta \in U, r > 0 \right \}$.
For a function $\vphi: (0, \infty) \rightarrow (0, \infty)$ and for $n \geq 2$ we write
$$ \subgraph_n(\vphi) = \{ (x,y) \in \RR \times \RR^{n-1} \, ; \, x > 0, |y| \leq \vphi(x) \} \subset \RR^n. $$
We consider two {\it quasi-net} conditions that a subset $L \subseteq \RR^n$ may satisfy:
\begin{enumerate}
\item[(QN1)] There exists a non-decreasing function $\vphi: (0, \infty) \rightarrow (0, \infty)$ with $\lim_{r \rightarrow \infty} \vphi(r) / \sqrt{r} = 0$
such that for any isometry $T: \RR^n \rightarrow \RR^n$, $$ L \cap T( \subgraph_n(\vphi) ) \neq \emptyset. $$
\item[(QN2)] For any non-empty, open sector $S \subseteq \RR^n$  there exists a sequence $(p_m)_{m \geq 1}$
with $p_m \in L \cap S$ for all $m$ such that
$$ \lim_{m \rightarrow \infty} \frac{|p_{m+1}|}{|p_m|} = 1 \qquad \text{and} \qquad \lim_{m \rightarrow \infty} |p_m| = \infty. $$
\end{enumerate}
It is clear that any net in $\RR^n$ satisfies conditions (QN1) and (QN2). A random instance of a Poisson process with uniform intensity in $\RR^n$ is
a discrete set satisfying (QN1) and (QN2), with probability one. Hence Theorem \ref{thm_12170}
and Theorem \ref{thm_1217_0} are particular cases of the following:

\begin{theorem}
Let $M$ be a complete, connected, $2$-dimensional Riemannian manifold.
Suppose that there exists a discrete set in $\RR^2$ which satisfies conditions (QN1) and (QN2)
and that embeds isometrically in $M$.
Then the manifold $M$ is flat and it is isometric to the Euclidean plane $\RR^2$.
\label{thm_1217}
\end{theorem}

\begin{theorem}
Let $M$ be a complete, connected, $n$-dimensional Riemannian manifold.
Suppose that there exists a discrete set in $\RR^n$ which satisfies condition (QN1)
and that embeds isometrically in $M$.
Then $M$ is diffeomorphic to $\RR^n$.
\label{thm_1217_}
\end{theorem}

\medskip The remainder of this paper is devoted almost entirely  to the proofs of Theorem \ref{thm_1217} and
Theorem \ref{thm_1217_}.
The key step in the proof of Theorem \ref{thm_1217} is to show that $M$ has no conjugate points.
This enables us to make contact with the developed mathematical literature
on the rigidity of Riemannian manifolds without conjugate points under topological assumptions, under curvature assumptions or
under isoperimetric assumptions. The relevant literature begins
with the works
of Morse and Hedlund \cite{MH} and Hopf \cite{H}, and continues with contributions by Bangert and Emmerich \cite{BE1, BE2}, Burago and Ivanov \cite{BI},
Burns and Knieper \cite{BK}, Busemann \cite{Bu}, Croke \cite{Cr1, Cr2} and others. At the final step of the argument
below we apply the equality case of the area growth inequality due to Bangert and Emmerich \cite{BE2},
whose beautiful proof is based on Hopf's method.

\medskip The mathematical literature pertaining to nets that approximate a Riemannian manifold
includes the analysis by Fefferman, Ivanov, Kurylev, Lassas and Narayanan \cite{F, F2},
and the works by Fujiwara \cite{Fu} and by Burago, Ivanov and Kurylev \cite{BIK} on approximating the spectrum and eigenfunctions of the Laplacian via a net.
These works are related to the useful idea of a {\it diffusion map}, as in Belkin and Niyogi \cite{BN}, Coifmann and Lafon \cite{CL} and Singer
\cite{S}.

\medskip All of the Riemannian manifolds below are assumed to be $C^2$-smooth, and all parametrizations of
geodesics are by arclength. Thus a geodesic here is always of unit speed. We write that a function
$f(t)$ is $o(t)$ as $t \rightarrow \infty$ if $f(t) / t$ tends to zero as $t \to \infty$.

\medskip
{\it Acknowledgements.} The second-named author would like to thank
Charles Fefferman for interesting discussions on possible Riemannian analogs of Whitney's extension problem,
and to Adrian Nachman for excellent explanations on the boundary rigidity problem and other inverse problems. Both authors thank Itai Benjamini
for his interest and encouragement. Supported by a grant from the Israel Science Foundation (ISF).

\section{Lipschitz functions}

We begin the proofs of Theorem \ref{thm_1217} and Theorem \ref{thm_1217_} with  some background on geodesics and Lipschitz functions.
Our standard reference for Riemannian geometry is Cheeger and Ebin \cite{CE}.

\medskip We work in a complete, connected,  Riemannian manifold $M$ with distance function $d$.
A {\it minimizing geodesic} is a curve $\gamma: I \rightarrow M$, where $I \subseteq \RR$ is an interval (i.e., a connected set), with
$$ d(\gamma(t), \gamma(s)) = |t-s| \qquad \qquad \text{for all} \ s,t \in I. $$
As is customary, our notation does not fully distinguish between the parametrized curve $\gamma: I \rightarrow M$
and its image $\gamma( I )$ which is just a subset of $M$, sometimes endowed with an orientation. It should be clear from the context whether we mean a parametrized
curve, or its image in $M$.

\medskip A curve $\gamma: I \rightarrow M$ is a geodesic if the interval $I$ may be covered by open
intervals on each of which $\gamma$ is a minimizing geodesic.
In the case where $I = \RR$ we say that the geodesic $\gamma$ is complete.
When $I = [0, \infty)$ or $I = (0, \infty)$ we say that $\gamma$ is a ray, and if $I \subseteq \RR$
is of finite length we say that $\gamma$ is a geodesic segment.
Since $M$ is complete, for any $x,y \in M$ there exists a minimizing geodesic segment connecting $x$ and $y$.
A minimizing geodesic ray cannot intersect a minimizing geodesic segment at more than one point unless they overlap, see \cite{CE}.

\medskip
Let $\gamma_m: I_m \rightarrow M \ (m=1,2,\ldots)$ be a sequence of geodesics. We say that the sequence converges to a geodesic $\gamma: I \rightarrow M$
if $I = \cup_{m \geq 1} \cap_{k \geq m} I_k$ and for any $t \in I$,
$$ \gamma_{m}(t)\xrightarrow{m\to\infty}\gamma(t). $$
In the case where $I_m = I$ for all $m$ the following holds: For any fixed $t_0 \in I$, the convergence $\gamma_m \longrightarrow \gamma$ is equivalent to the requirement that
$$ \gamma_m(t_0) \xrightarrow{m\to\infty} \gamma(t_0) \qquad \text{and} \qquad \dot{\gamma}_m(t_0) \xrightarrow{m\to\infty} \dot{\gamma}(t_0). $$
Here $\dot{\gamma}(t_0) \in T_{\gamma(t_0)} M$ is the tangent vector to the geodesic $\gamma$, and $T_p M$ is the tangent space to $M$ at
the point $p \in M$. A sequence of  unparametrized geodesics is said to converge if its geodesics may be parametrized to yield a converging sequence in the above sense.

\medskip
The continuity of the distance function
implies that the limit of a converging sequence of minimizing geodesics, is itself a minimizing geodesic.
Any sequence of geodesics passing through a fixed point $x \in M$, has a convergent subsequence.
We say that a sequence of points in $M$ tends to infinity if any
compact $K\subseteq M$ contains only finitely many points from the sequence.
When $x, x_1,x_2,\ldots$ are points in $M$ with $x_m \longrightarrow \infty$ and $\gamma_m$ is a minimizing geodesic
connecting $x$ with $x_m$, the sequence $(\gamma_m)_{m \geq 1}$ has a subsequence
that converges to a minimizing geodesic ray emanating from $x$.

\medskip
Lipschitz functions are somewhat ``dual'' to curves and geodesics in the following sense:
Any rectifiable curve between $x$ and $y$ provides an upper bound for the distance
$d(x,y)$. On the other hand, a $1$-Lipschitz function $f: M \rightarrow \RR$ is a function that satisfies $|f(x) - f(y)| \leq d(x,y)$
for all $x,y \in M$, and hence it provides lower bounds for the distance $d(x,y)$. When $f: M \rightarrow \RR$ is $1$-Lipschitz and $\gamma:I \rightarrow M$ is a geodesic,
\begin{equation} |f(\gamma(t)) - f(\gamma(s))| \leq d(\gamma(t), \gamma(s)) \leq |t-s| \qquad \text{for all} \ s,t \in I.
\label{eq_10025} \end{equation}
We say that the geodesic $\gamma$ is a {\it transport curve} of the $1$-Lipschitz function $f$ if
\begin{equation}  f(\gamma(t)) - f(\gamma(s)) = t-s \qquad \qquad \text{for all} \ s,t \in I. \label{eq_1003} \end{equation}
Thus, the function $f$ grows with unit speed along a transport curve.
This terminology comes from the theory of optimal transport, see e.g. Evans and Gangbo \cite{EG} or \cite[Section 2.1]{K}.
If $I = \RR$ then we say that the transport curve $\gamma$ is a {\it transport line} and if the transport curve $\gamma$ is a geodesic ray then $\gamma$ is called a {\it transport ray}. It follows from (\ref{eq_10025}) and (\ref{eq_1003})
that any transport curve is a minimizing geodesic.

\medskip When $\gamma: I \rightarrow M$ is a transport curve of a $1$-Lipschitz function $f$, the function $f$ is differentiable at $\gamma(t)$ for all $t$ in the interior of the interval $I \subseteq \RR$, as proven in Feldman and McCann \cite[Lemma 10]{FM}.
For any $t \in I$ such that $f$ is differentiable at $\gamma(t)$, we have
\begin{equation}
\nabla f(\gamma(t)) = \dot{\gamma}(t). \label{eq_825} \end{equation}
Indeed, it follows from (\ref{eq_1003}) that $\langle \nabla f(\gamma(t)), \dot{\gamma}(t) \rangle = 1$, where $\langle \cdot, \cdot \rangle$
and $| \cdot |$ are the Riemannian scalar product and norm in $T_{\gamma(t)} M$, and hence (\ref{eq_825}) follows
as  $|\nabla f(\gamma(t))| \leq 1$ and $|\dot{\gamma}(t)| = 1$.

\begin{lemma} Let $f: M \rightarrow \RR$ be a $1$-Lipschitz function. Suppose that $\gamma_1$ is a transport line
of $f$ and that $\gamma_2$ is a transport curve of $f$ with $\gamma_1 \cap \gamma_2 \neq \emptyset$. Then $\gamma_2 \subseteq \gamma_1$.
In particular, if $\gamma_2$ is a transport line as well, then the geodesics $\gamma_1$ and $\gamma_2$ coincide.
\label{lem_1005}
\end{lemma}

\begin{proof} Since both $\gamma_{1}$ and $\gamma_{2}$ are minimizing geodesics passing through a point $x\in\gamma_{1}\cap\gamma_{2}$ in the direction $\nabla f(x)$, necessarily $\gamma_2 \subseteq \gamma_1$.
\end{proof}

\begin{corollary} Let $f: M \rightarrow \RR$ be a $1$-Lipschitz function, let $\gamma$ be a transport line of $f$ and
fix a point $x \in \gamma$.
Then for any $y \in M$,
$$
y \in \gamma \qquad \Longleftrightarrow \qquad |f(y) - f(x)| = d(x,y). $$ \label{cor_230}
\end{corollary}

\begin{proof}
Assume that $|f(y) - f(x)| = d(x,y)$ and connect $y$ to $x$ by a minimizing geodesic
$\eta$. The geodesic $\eta$ is necessarily a transport curve of $f$ passing through the point $x \in \gamma$. We conclude from Lemma \ref{lem_1005} that $y \in \eta \subseteq \gamma$. For the other direction, since $x,y \in \gamma$
while $\gamma$ is a transport curve of $f$, it follows from (\ref{eq_10025}) and (\ref{eq_1003}) that $|f(y) - f(x)| = d(x,y)$.
\end{proof}

The first example of  a $1$-Lipschitz function in $M$ is the distance function
$x \mapsto d(p,x)$ from  a given point $p \in M$. Any minimizing geodesic segment connecting $p$
to a point $y \in M$ is a transport curve of this distance function.
The second example is the Busemann function of a minimizing geodesic $\gamma: [t_0, \infty) \rightarrow M$, defined
as
\begin{equation}
B_\gamma(x) = \lim_{t \rightarrow \infty} \left[ t - d(\gamma(t), x) \right]
= \sup_{t \geq t_0} \left[ t - d(\gamma(t), x)  \right]. \label{eq_1014}
\end{equation}
Our definition of $B_{\gamma}$ differs by a sign from the convention in Ballman, Gromov and Schroeder \cite{BG} and in Busemann \cite{Bu}.
It is well-known that the limit in (\ref{eq_1014}) always exists, since $t - d(\gamma(t), x)$ is non-decreasing in $t$ and bounded from above by $d(\gamma(t_0), x) + |t_0|$. Moreover, the function $B_{\gamma} : M \rightarrow \RR$ is a $1$-Lipschitz function.
Thanks to our sign convention, any minimizing geodesic  $\gamma: [t_0, \infty) \rightarrow M$ is a
transport curve of the $1$-Lipschitz function $B_{\gamma}$.

\medskip If a sequence of $1$-Lipschitz functions $f_{m}: M \rightarrow \RR \ \ (m\geq 1)$ converges pointwise
as $m \rightarrow \infty$ to a limit function $f: M \rightarrow \RR$, then $f$  is $1$-Lipschitz. Moreover, the convergence is locally uniform
by the Arzela-Ascoli theorem. We will frequently use the following fact: For a continuous $f$, the convergence $f_{m} \longrightarrow f$
is locally uniform if and only if whenever $M\ni x_{m} \longrightarrow x$, also $f_{m}(x_{m}) \longrightarrow f(x)$.

\begin{lemma}
Let $f_{m}:M\to\RR$ be a sequence of $1$-Lipschitz functions that converges pointwise to $f: M \to \RR$, and let $\gamma_{m}:I_{m}\to M$
be a sequence of geodesics converging to $\gamma:I\to M$ such that $\gamma_{m}$ is a transport curve of $f_m$ for all $m$.
Then $\gamma$ is a transport curve of $f$.
\label{lem_125}
\end{lemma}

\begin{proof}
Let $s,t\in I$. Since $\gamma_{m}$ is a transport curve of $f_m$, for a sufficiently large $m$,
\begin{equation}
f_{m}(\gamma_{m}(t))-f_{m}(\gamma_{m}(s))=t-s.
\label{eq_125}
\end{equation}
Since the convergence $f_{m} \longrightarrow f$ is locally uniform
in $M$, we have $f_{m}(\gamma_{m}(t))\longrightarrow  f(\gamma(t))$ for
all $t\in I$. Letting $m\to\infty$ in (\ref{eq_125}) yields
$f(\gamma(t)) - f(\gamma(s)) = t-s $. \end{proof}

We say that a $1$-Lipschitz function $f: M \rightarrow \RR$ {\it induces a foliation by transport lines}, or in short {\it foliates}, if
for any $x \in M$ there exists a transport line of $f$ that contains $x$. By
 Lemma \ref{lem_1005}, in this case $M$ is the disjoint union of the
 transport lines of $f$. When a function $f$ foliates, it is differentiable everywhere in $M$.
The following proposition describes a way to produce $1$-Lipschitz functions that foliate.
For $x\in M$, we denote the cut-locus of $x$ by $\text{cut}(x) \subseteq M$. See \cite{CE} for information
about the cut-locus.

\begin{proposition}
Let $(y_{m})_{m \geq 1}$ be a sequence of points in $M$ tending to infinity and let $(C_m)_{m \geq 1}$ be  real numbers.
Denote
$$ f_m(x) = C_m - d(x,y_m) \qquad \qquad \qquad (x \in M, m \geq 1) $$
and assume that $f_m \longrightarrow f$ pointwise in $M$ as $m \rightarrow \infty$.
Then:
\begin{enumerate}
\item[(i)] If $\text{cut}(y_{m})=\emptyset$ for all $m$, then $f$ foliates.
\item[(ii)] Suppose that $f$ foliates and fix $x \in M$. Then for any sequence of
minimizing geodesics $\gamma_m: [0, d(x,y_m)] \rightarrow M$
with $\gamma_m(0) = x$ and $\gamma_m(d(x,y_m)) = y_m$,
$$ \dot{\gamma}_{m}(0)\xrightarrow{m\to\infty}\nabla f(x).$$
\end{enumerate}
\label{lem_519}
\end{proposition}

\begin{proof} Fix $x \in M$ and set $r_m = d(x, y_m)$. Since $M$ is complete and $y_m \longrightarrow \infty$,
necessarily $r_m \longrightarrow \infty$.  We proceed with the proof of (ii). \begin{enumerate}

\item[(ii)] It suffices to prove that any convergent subsequence of $(\dot{\gamma}_m(0))_{m \geq 1}$
tends to $\nabla f(x)$. We may thus pass to a subsequence, and assume that $\dot{\gamma}_m(0) \longrightarrow v$
for a unit vector $v \in T_x M$. Our goal is to prove that $v = \nabla f(x)$. Passing to a further subsequence,
 we may additionally  assume that the limit $$ \gamma=\lim_{m \rightarrow \infty}\gamma_{m} $$
exists. For any $t>0$,
we know that $t\in[0,r_{m}]$ for a sufficiently large $m$, which implies
that $\gamma$ is defined on $[0,\infty)$. The geodesic $\gamma_{m}$ is a transport curve of $f_m$ for any $m$, since for $t,s\in  [0, r_m]$ we have
$$ f_{m}(\gamma_{m}(t))-f_{m}(\gamma_{m}(s))=d(\gamma_{m}(s),y_{m})-d(\gamma_{m}(t),y_{m})=t-s.
$$
Lemma \ref{lem_125} shows that $\gamma$ is a transport curve of $f$. Since $f$ foliates, it is differentiable at $x$, and therefore $\dot{\gamma}_{m}(0) \longrightarrow \dot{\gamma}(0)=\nabla f(x)$ where we used (\ref{eq_825}) in the last passage.

\item[(i)] In order to show that $f$ foliates, we need to find
a transport line of $f$ that passes through $x$.
Since $\text{cut}(y_{m})=\emptyset$, we may write $\gamma_{m}:(-\infty, r_m] \to M$
for the unique minimizing geodesic with $$ \gamma_m(0) = x \qquad \text{and} \qquad \gamma_m(r_m) = y_m. $$
 Passing
to a subsequence, we may assume that $\gamma_{m}\longrightarrow\gamma$ for a
minimimizing geodesic $\gamma$ with $\gamma(0)=x$. Since $r_m \longrightarrow \infty$, the geodesic $\gamma$ is complete.
The geodesic $\gamma_m$ is a transport curve of $f_m$ for any $m$, and from
Lemma \ref{lem_125} we conclude  that $\gamma$ is a transport line
of $f$ that passes through $x$. Hence $f$ foliates. \qedhere
 \end{enumerate}
\end{proof}

\begin{lemma} Let $V$ be a metric space, and assume that with any $v \in V$ we associate a $1$-Lipschitz function $f_v: M \rightarrow \RR$.
Suppose that $f_v$ foliates for any $v \in V$, and that $f_v(x)$ varies continuously with $v \in V$ for any fixed $x \in M$. Then the map
$$ (x, v) \mapsto \nabla f_v(x) $$
is continuous in $M \times V$.
\label{lem_200}
\end{lemma}

\begin{proof}
It suffices to show that for any sequence $M\times V\ni(x_{m},v_{m})\longrightarrow(x,v)\in M\times V$,
the sequence $( \nabla f_{v_{m}}(x_{m}) )_{m \geq 1}$ has a subsequence converging to $\nabla f_{v}(x)$.
Abbreviate $f_{m}=f_{v_{m}}$
and $f=f_{v}$, so that $f_{m} \longrightarrow f$ locally uniformly by the Arzela-Ascoli theorem.
For each $m$ consider the transport line $\gamma_{m}:\RR\to M$ of $f_{m}$ which satisfies
$$
\gamma_m(0) = x_m \qquad \text{and} \qquad \dot{\gamma}_m(0) = \nabla f_{m}(x_m).
$$
Since $\gamma_m(0) \longrightarrow x$ as $m \rightarrow \infty$, we may pass to a subsequence and assume that $\gamma_{m} \longrightarrow \gamma$ for some minimizing geodesic $\gamma:\RR\to M$ with
$$
\gamma(0) = x \qquad \text{and} \qquad \dot{\gamma}(0)=\lim_{m\to\infty}\dot{\gamma}_{m}(0)=\lim_{m\to\infty}\nabla f_{m}(x_{m}).
$$
Lemma \ref{lem_125} states that $\gamma$ is a transport line
of $f$. In particular $\dot{\gamma}(0)=\nabla f(x)$ by (\ref{eq_825}).
\end{proof}

A corollary of Lemma \ref{lem_200} is that any $1$-Lipschitz function that foliates is a $C^1$-function.
These functions are actually $C^{1,1}$-smooth, see \cite[Theorem 2.1.13]{K}.

\section{Directional drift to infinity}

From now on and until the end of Section \ref{sec_no_conj}, our standing assumptions
are the assumptions of Theorem \ref{thm_1217_}.
We thus work in a complete, connected, $n$-dimensional Riemannian manifold $M$, with  $n \geq 2$.
We assume that $\iota: L \rightarrow M$ is an isometric embedding for the discrete set
$$ L \subseteq \RR^n, $$
that satisfies condition (QN1). Translating the discrete set $L$ does not alter the validity of condition (QN1)
or condition (QN2), hence we may translate $L$ and assume for convenience that
$$ 0 \in L.
$$
 For ease of reading, and with a slight abuse of notation, we identify between a point
 $a \in L \subseteq \RR^n$ and its image $\iota(a) \in M$. Thus we think of $L$ as a subset of $M$, and the assumption that $\iota$ is an isometric embedding translates to
\begin{equation}  d(x,y) = |x-y| \qquad \qquad \qquad \text{for all} \ x, y \in L. \label{eq_600} \end{equation}
Note that for  $p \in L \subseteq M$, we may speak of the Euclidean norm $|p| = \sqrt{\sum_i p_i^2} = d(0,p)$ and of
the scalar product $\langle p, v \rangle = \sum_i p_i v_i$ for $v \in \RR^n$.
Given $p \in L$ we write $d_{p}:M\to\RR$ for the function
$$ d_p(x) = d(p, 0) - d(p, x)$$
which is a $1$-Lipschitz function that vanishes at $0 \in L \subseteq M$. The following notion is
 in the spirit of the ``ideal boundary'' of a Hadamard manifold (see, e.g., \cite{BG}).

\begin{definition}
Let $v \in S^{n-1}$ and let $B: M \rightarrow \RR$ be a $1$-Lipschitz function.
We write that $B \in \partial_v M$ if
$$ B(p) = \langle p, v \rangle \qquad \qquad \text{for all} \ p \in L.
$$
\label{def_409}
\end{definition}

We say that a sequence of points $p_m \in \RR^n \ \ (m=1,2,\ldots)$ is {\it drifting in the direction of} $v \in S^{n-1}$,
and we write $ p_m \rightsquigarrow v $, if
\begin{equation}  |p_m| \xrightarrow{m \to \infty} \infty \qquad \text{and} \qquad \frac{p_m}{|p_m|} \xrightarrow{m \to \infty} v. \label{eq_1038}
\end{equation}

\begin{proposition}
Let $v\in S^{n-1}$, $B:M\to\RR$, and let $p_{m} \in L$ satisfy $ p_m \rightsquigarrow v $. Assume that $d_{p_{m}} \longrightarrow B$ pointwise as $m\to\infty$. Then $B\in \partial_v M$.
\label{lem_149}
\end{proposition}

\begin{proof} The function $B$ is $1$-Lipschitz, being the pointwise limit of a sequence of $1$-Lipschitz functions.
By (\ref{eq_600}), for any $q \in L$,
$$	B(q)=\lim_{m\to\infty}d_{p_{m}}(q)=\lim_{m\to\infty}[d(p_{m},0)-d(p_{m},q)]=\lim_{m\to\infty}[\vert p_{m}\vert-\vert p_{m}-q\vert]=\langle q,v\rangle,
$$
where in the last passage we used the following computation in Euclidean geometry:
$$ \lim_{m\to\infty}[\vert p_{m}\vert-\vert p_{m}-q\vert]  = \lim_{m \to \infty} \frac{2 \langle p_m, q \rangle - |q|^2}{|p_m| + |p_m-q|} =
\lim_{m \to \infty} \left \langle \frac{p_m}{|p_m|}, q \right \rangle = \langle v, q \rangle. $$
 Thus $B\in \partial_v M$.
\end{proof}

Let $p_m \in \RR^n \ \ (m=1,2,\ldots)$ be a sequence of points and let $v \in S^{n-1}$. We say that $p_m \rightsquigarrow v$ {\it narrowly}
if
\begin{equation}   |p_m| \xrightarrow{m \rightarrow \infty} \infty  \qquad \text{and} \qquad |p_m| - \langle p_m, v \rangle \xrightarrow{m \rightarrow \infty}  0. \label{eq_151}
\end{equation}
Clearly (\ref{eq_151}) implies (\ref{eq_1038}).  Two properties of narrow drift are summarized in the following:

\begin{lemma}
For any $v \in S^{n-1}$ there exists a sequence $(p_m)_{m \geq 1}$ in $L$ with $p_m \rightsquigarrow v$ narrowly.
Moreover, for any such sequence and for any $p \in L$, also $p_m - p \rightsquigarrow v$ narrowly.
\label{lem_121}
\end{lemma}

\begin{proof}
Let $(p_m)_{m \geq 1}$ be a sequence in $L$ with $|p_m| \longrightarrow \infty$, and write $r_m = |p_m|$. Since
for any $v \in S^{n-1}$, $$
\frac{\vert p_{m}-r_{m}v\vert}{\sqrt{r_{m}}}=\sqrt{2(r_{m}-\langle p_{m},v\rangle)},
$$
condition (\ref{eq_151}) is equivalent to
\begin{equation}
\frac{|p_m - r_m v |}{\sqrt{r_m}} \xrightarrow{m \rightarrow \infty} 0.
\label{eq_1106} \end{equation}
We thus need to find $p_m \in L$  with $p_m \longrightarrow \infty$ such that (\ref{eq_1106}) holds true.
Since $L \subseteq \RR^n$ satisfies (QN1), there exists a non-decreasing function $\vphi: (0, \infty) \rightarrow (0, \infty)$ with
$\vphi(r) = o(\sqrt{r})$ as $ r \rightarrow \infty$ such that for any isometry $T: \RR^n \rightarrow \RR^n$,
\begin{equation}  L \cap T(\subgraph_n(\vphi)) \neq \emptyset. \label{eq_432} \end{equation}
Let $U: \RR^n \rightarrow \RR^n$ be a linear orthogonal transformation that maps the standard unit vector $e_1$ to the unit vector $v$,
and let $T_m( x) = U(x) + m \cdot v$ be an isometry of $\RR^n$. By applying (\ref{eq_432}) we conclude
that for any $m \geq 1$ there exists a point $p_m \in L$ such that $q_m = p_m - m \cdot v$ satisfies $\langle q_m, v \rangle > 0$ and
\begin{equation}  |\Proj_{v^{\perp}} q_m| \leq \vphi( \langle q_m, v \rangle ), \label{eq_417} \end{equation}
where the linear map $\Proj_{v^{\perp}}: \RR^n \rightarrow \RR^n$ is the orthogonal projection on the hyperplane orthogonal to $v$ in $\RR^n$.
Note that $p_m \longrightarrow \infty$ since $|p_m| \geq \langle p_m, v \rangle \geq m$. As $\vphi$ is non-decreasing, according to (\ref{eq_417}),
$$  |\Proj_{v^{\perp}} p_m| = |\Proj_{v^{\perp}} q_m| \leq \vphi( \langle q_m, v \rangle )
\leq \vphi( \langle p_m, v \rangle ) = o( \sqrt{\langle p_m, v \rangle} ) = o(\sqrt {r_m}),  $$
as $r_m = |p_m|$. Since $r_m \leq |\Proj_{v^{\perp}} p_m| + |\langle p_m, v \rangle| = \langle p_m, v \rangle + o( \sqrt {r_m})$,
$$ |p_m - r_m v| \leq |\Proj_{v^{\perp}} p_m| + |\langle p_m, v \rangle - r_m|  = o(\sqrt {r_m}) + (r_m - \langle p_m, v \rangle) = o(\sqrt{r_m}), $$
proving (\ref{eq_1106}). Moreover, given any sequence $p_m \rightsquigarrow v$ narrowly, it follows from (\ref{eq_1106}) that for $\tilde{p}_m = p_m - p,
r_m = |p_m|$ and $s_m = |\tilde{p}_m|$,
$$ |\tilde{p}_m - s_m v| \leq |p_m - r_m v|  + 2 |p| = o(\sqrt{r_m}) + 2 |p| = o(\sqrt{s_m}) $$
and hence $\tilde{p}_m \rightsquigarrow v$ narrowly as well.
\end{proof}

Assumption (QN1) in Theorem \ref{thm_1217} and in Theorem \ref{thm_1217_}
may actually be replaced by assumption (QN1'), which is the condition
that for any $v \in S^{n-1}$ there exists a sequence $(p_m)_{m \geq 1}$ in $L$ with $p_m \rightsquigarrow v$ narrowly.
We also note here that neither (QN1) implies (QN2) nor (QN2) implies (QN1).

\begin{lemma}
The set $\partial_{v}M$ is non-empty for any $v\in S^{n-1}$. In fact, for any sequence $L \ni p_{m} \rightsquigarrow v$ there exists a subsequence such that $d_{p_{m_{k}}} \longrightarrow B$ for some $B \in \partial_{v} M$.
\label{lem_1045}
\end{lemma}

\begin{proof} It follows from Lemma \ref{lem_121} that there exist points $L \ni p_m \rightsquigarrow v$.
The sequence $(d_{p_{m}})_{m \geq 1}$ consists of $1$-Lipschitz
functions vanishing at $0$. By the Arzela-Ascoli theorem, there
exists a subsequence $(d_{p_{m_{k}}})_{k \geq 1}$ such that $d_{p_{m_{k}}} \longrightarrow B$
locally uniformly for some $1$-Lipschitz function $B:M\to\mathbb{R}$.
By Proposition \ref{lem_149}, we have that $B\in\partial_{v}M$.
\end{proof}

From Definition \ref{def_409} it follows that for any $v \in S^{n-1}$,
\begin{equation}  \partial_{-v} M = -\partial_v M := \{ -B \, ; \, B \in \partial_v M \}. \label{eq_639}
\end{equation}
The next proposition is the reason for introducing the notion of a narrow drift. It produces complete minimizing
geodesics through points of $L$ that interact nicely with $\partial_v M$.

\begin{proposition}
Let $p \in L, v \in S^{n-1}$ and assume that $(p_m^+)_{m \geq 1}$ is a sequence in $L$ with $p_m^+ \rightsquigarrow v$ narrowly,
while $(p^-_m)_{m \geq 1}$ is a sequence in $L$ satisfying $p^-_m \rightsquigarrow -v$ narrowly.

\medskip For any $m$, let $\gamma^{\pm}_m$ be a minimizing geodesic connecting $p$ and $p_m^{\pm}$.
Assume that $\lim_m \gamma^{\pm}_m= \gamma^{\pm}$
for a geodesic ray $\gamma^{\pm}: [0, \infty) \rightarrow M$ with $\gamma^{\pm}(0)=  p$. Then the concatenation
$\gamma = \gamma^+ \cup \gamma^-$ with parametrization
\begin{equation}  \gamma(t) = \left \{ \begin{array}{cc} \gamma^+(t) & t \geq 0 \\
\gamma^-(-t) & t \leq 0 \end{array} \right. \label{eq_700} \end{equation}
 is a transport line of $B$, for any $B \in \partial_v M$.
\label{lem_454a}
\end{proposition}

\begin{proof}  Set $r_m^{\pm} = |p_m^\pm - p|$.
Then $r_m^{\pm} \longrightarrow \infty$  by (\ref{eq_151}).
We parametrize our geodesics as
$\gamma_m^{\pm}: [0, r_m^{\pm}] \rightarrow M$ with
$$ \gamma_m^{\pm}(0) = p \qquad \text{and} \qquad \gamma_m^{\pm}(r_m^{\pm}) = p_m^\pm. $$
By our assumption, $\gamma_m^{\pm} \longrightarrow \gamma^{\pm}$
where $\gamma^{\pm}: [0, \infty) \rightarrow M$
is a geodesic with $\gamma^{\pm}(0) = p$. Fix $t > 0$ and $B \in \partial_v M$. In order to show that $\gamma$, as defined in (\ref{eq_700}), is a transport line of $B$, it suffices to show that
\begin{equation}
B(\gamma^+(t)) - B(\gamma^+(0)) = B(\gamma^-(0)) - B(\gamma^-(t)) = t.
\label{eq_1520} \end{equation}
Since $p_m^{\pm} \rightsquigarrow \pm v$  narrowly, it follows from Lemma \ref{lem_121} that $p_m^{\pm} - p \rightsquigarrow \pm v$  narrowly as well, i.e.
$$ |p_m^{\pm} - p| - \langle p_m - p, \pm v \rangle \xrightarrow{ m \to \infty} 0. $$
For a sufficiently large $m$, we know that $t \leq \min \{ r_m^+, r_m^- \}$. Since $B$ is $1$-Lipschitz, by (\ref{eq_10025}),
\begin{align} \label{eq_423_}
 B(\gamma_{m}^+(t))-B(\gamma_{m}^+(0))-t & \geq B(\gamma_{m}^+(r_m^+))-B(\gamma_{m}^+(0)) - r_m^+
\\& = B(p_m^+) - B(p) - r_m^+
=\langle p_m^+ - p ,v\rangle-\vert p_{m}^+- p \vert\xrightarrow{m \to \infty}  0. \nonumber
\end{align}
Similarly, 
\begin{align*}
-B(\gamma^-_{m}(t))+B(\gamma^-_{m}(0))-t &\geq -B(\gamma^-_{m}(r^-_m))+B(\gamma^-_{m}(0)) - r^-_m
\\&
=\langle p_m^- - p ,-v\rangle-\vert p^-_{m}- p \vert\xrightarrow{m \to \infty}  0.
\end{align*}
Since $\gamma_{m}^{\pm} \longrightarrow \gamma^{\pm}$, by taking the limit $m\to\infty$
we obtain
$$
B(\gamma^+(t))-B(\gamma^+(0))- t\geq 0
\qquad \text{and} \qquad -B(\gamma^-(t))+B(\gamma^-(0))- t\geq 0.
$$
The reverse inequalities are trivial by (\ref{eq_10025}), and hence (\ref{eq_1520}) follows.
\end{proof}

\begin{lemma} Let $p \in L$ and $v \in S^{n-1}$.
Let $(p_m)_{m \geq 1}$ and $(q_m)_{m \geq 1}$ be sequences in $L$ such that $p_m \rightsquigarrow v$ and $q_m \rightsquigarrow v$ with
at least one of the drifts being narrow. Let $\gamma_m$ be a minimizing geodesic from $p$ to $p_m$
and let $\eta_m$ be a minimizing geodesic from $p$ to $q_m$. Then there exists a geodesic ray $\gamma$ with $\gamma(0) = p$ such that both
$ \gamma_m \longrightarrow \gamma$ and $\eta_m \longrightarrow \gamma$.
\label{lem_454}
\end{lemma}

\begin{proof} Passing to convergent subsequences, we may assume that $\gamma_m \longrightarrow \gamma$
and $\eta_m \longrightarrow \eta$ for some geodesic rays $\gamma$ and $\eta$ with $\gamma(0) = \eta(0) = p$, and our goal is to prove that $\gamma \equiv \eta$.

\medskip Assume that the drift $p_m \rightsquigarrow v$ is narrow.
By Lemma \ref{lem_1045}, we may pass to a subsequence, and assume that $d_{q_m} \longrightarrow B$ for a certain $ B \in \partial_v M. $
Since $\eta_m$ is a minimizing geodesic connecting $p$ and $q_m$,
it is a transport curve of $d_{q_m}$. Since $d_{q_m} \longrightarrow B$ and $\eta_m \longrightarrow \eta$,
Lemma \ref{lem_125} implies that the geodesic ray $\eta$ is a transport ray of $B$.

\medskip
According to Lemma \ref{lem_121}, there exists a sequence $(\tilde{p}_m)_{m \geq 1}$ in $L$ with $\tilde{p}_m \rightsquigarrow -v$
narrowly. Passing to a subsequence, we may assume that
$\tilde{\gamma}_m$, a minimizing geodesic from $p$ to $\tilde{p}_m$, converges as $m \rightarrow \infty$
to a geodesic ray $\tilde{\gamma}$ with $\tilde{\gamma}(0) = p$. Recall that the drift $p_m \rightsquigarrow v$ is narrow and
that $\gamma_m \longrightarrow \gamma$.
By Proposition \ref{lem_454a}, the concatenation $\hat{\gamma} = \gamma \cup \tilde{\gamma}$
is a transport line of $B$ with the parametrization
\begin{equation}  \hat{\gamma}(t) = \left \{ \begin{array}{cc} \gamma(t) & t \geq 0 \\
\tilde{\gamma}(-t) & t \leq 0 \end{array} \right. \label{eq_700_} \end{equation}
Thus $\hat{\gamma}$ is a transport line of $B$ and  $\eta$ is a transport ray of
$B$, both passing through the point $p$. Lemma \ref{lem_1005} implies that
\begin{equation}  \eta \subseteq \hat{\gamma} = \gamma \cup \tilde{\gamma}. \label{eq_710_} \end{equation}
The three curves $\eta, \gamma$ and $\tilde{\gamma}$ are geodesic rays emanating from $p$.
It thus follows from (\ref{eq_710_}) that either $\eta = \gamma$ or else $\eta = \tilde{\gamma}$.
However, $\eta: [0, \infty) \rightarrow M$ and $\gamma: [0, \infty) \rightarrow M$ are transport rays of $B$ unlike $\tilde{\gamma}: [0, \infty) \rightarrow M$, as follows from (\ref{eq_700_}), hence $\eta \equiv \gamma$.
\end{proof}

We write $S_x M = \{ u \in T_x M \, ; \, |u| = 1 \}$ for the unit tangent sphere at the point $x \in M$.

\begin{proposition}
Fix $p \in L$. Then with any $v \in S^{n-1}$ there is a unique way to associate
a complete minimizing geodesic $\gamma_{p,v}: \RR \rightarrow M$ with $\gamma_{p,v}(0) = p$ such that the following hold:
\begin{enumerate}
\item[(i)] For any $v \in S^{n-1}$, if $(p_m)_{m \geq 1}$ is a sequence in $L$ with $p_m \rightsquigarrow v$, and $\gamma_m$
is a minimizing geodesic from $p$ to $p_m$, then $\gamma_m$ tends to the geodesic ray $\gamma_{p,v}([0, \infty))$ as $m \rightarrow \infty$.
\item[(ii)] The map $S^{n-1} \ni v \mapsto \dot{\gamma}_{p,v}(0) \in S_p M$ is odd, continuous and onto.
\item[(iii)] For any $v \in S^{n-1}$ and $B \in \partial_v M$, the minimizing geodesic $\gamma_{p,v}$ is a transport line of $B$.
\end{enumerate}
\label{prop_1155}
\end{proposition}

\begin{proof} For $v \in S^{n-1}$ we apply Lemma \ref{lem_121} and select a sequence $(q_m)_{m \geq 1} = (q_m^{(v)})_{m \geq 1}$ in $L$
with $q_m \rightsquigarrow v$ narrowly. Let $\eta_m = \eta_m^{(v)}$ be a minimizing geodesic segment with
\begin{equation}  \eta_m(0) = p \qquad \text{and} \qquad \eta_m(d(p, q_m)) = q_m. \label{eq_1010_} \end{equation}
Lemma \ref{lem_454}  implies that $\eta_m \longrightarrow \eta$ for a geodesic ray $\eta: [0, \infty) \rightarrow M$
with $\eta(0) = p$. We now define
\begin{equation}  \gamma_{p,v}(t) = \eta(t) \qquad \qquad \qquad (t \geq 0). \label{eq_1014_}
\end{equation}
Lemma \ref{lem_454} also states that whenever $L \ni p_m \rightsquigarrow v$,
a
minimizing geodesic $\gamma_m$ from $p$ to $p_m$
tends to the geodesic ray $\gamma_{p,v}([0, \infty))$ as $m \rightarrow \infty$. Thus (i) holds true. The geodesic ray $\gamma_{p,v}$ is defined in (\ref{eq_1014_})
for all $v \in S^{n-1}$, but only for $t \geq 0$. We extend this definition by setting
\begin{equation}  \gamma_{p,v}(-t) := \gamma_{p, -v}(t) \qquad \qquad \qquad \text{for all} \ t \geq 0. \label{eq_1024} \end{equation}
Since $q_m^{(v)} \rightsquigarrow v$ narrowly and $q_m^{(-v)} \rightsquigarrow -v$ narrowly, Proposition \ref{lem_454a} shows that
$\gamma_{p, v}: \RR \rightarrow M$ is a transport line of $B$ for any $B \in \partial_v M$. Thus (iii) is proven.
Since $\partial_v M \neq \emptyset$ by Lemma \ref{lem_1045}, the complete geodesic $\gamma_{p,v}$ is therefore minimizing. It is clear from our construction
that $\gamma_{p,v}$ is uniquely determined by requirement (i).

\medskip All that remains is to prove (ii). The map $F(v) = \dot{\gamma}_{p,v}(0) $ from $S^{n-1}$ to $S_p M$ is odd according to (\ref{eq_1024}).
Let us prove its continuity. To this end, suppose that $S^{n-1} \ni v_m \longrightarrow v$,
and our goal is to prove that $\gamma_{p, v_m} \longrightarrow \gamma_{p,v}$.
For each fixed $m \geq 1$ we know that as $k \rightarrow \infty$,
$$ q_k^{(v_m)} \longrightarrow \infty, \qquad \frac{q_k^{(v_m)}}{|q_k^{(v_m)}|} \longrightarrow v_m
\qquad \text{and} \qquad \eta_k^{(v_m)} \longrightarrow \gamma_{p,v_m}. $$
Hence for any $m$ there exists $k_m$ such that $|q_{k_m}^{(v_m)}| \geq m$ while
\begin{equation}  \left| \frac{q_{k_m}^{(v_m)}}{|q_{k_m}^{(v_m)}|} - v_m \right| \leq \frac{1}{m}
\qquad \text{and} \qquad \left| \dot{\eta}_{k_m}^{(v_m)}(0) - \dot{\gamma}_{p, v_m}(0) \right| \leq \frac{1}{m}. \label{eq_1126}
\end{equation}
Since $v_m \longrightarrow v$ we learn from (\ref{eq_1126}) that $p_m := q_{k_m}^{(v_m)}$ satisfies $p_m \rightsquigarrow v$.
The minimizing geodesic segment $\gamma_m := \eta_{k_m}^{(v_m)}$ connects the point $p$ to $p_m$, according to (\ref{eq_1010_}).
From (i) we thus conclude that $\dot{\gamma}_m(0)$ converges to $\dot{\gamma}_{p,v}(0)$
as $m \rightarrow \infty$. Consequently we obtain from (\ref{eq_1126})  that $$ \gamma_{p, v_m} \xrightarrow{m \to \infty} \gamma_{p,v}. $$
Hence
the map  $F(v) = \dot{\gamma}_{p,v}(0) $ is continuous.
Recall that the unit tangent sphere $S_{p} M$ is diffeomorphic to $S^{n-1}$.
The Brouwer degree of $F$ as a continuous, odd map from $S^{n-1}$ to $S_{p} M$ is an odd number. In particular the degree is non-zero,
and hence $F$ is onto.
\end{proof}

\section{Geodesics through $L$-points}

Proposition \ref{prop_1155} admits the following corollary,
which completes the proof of Theorem \ref{thm_1217_}.

\begin{corollary}
The manifold $M$ is diffeomorphic to $\RR^n$.
In fact, for any $p \in L$, the exponential map $\exp_{p}: T_{p} M \rightarrow M$ is a diffeomorphism and
all geodesics passing through $p$ are minimizing.

\medskip Moreover, for any geodesic ray $\gamma$ that emanates from $p$ there is a sequence $(p_m)_{m \geq 1}$ in $L$
with $p_m \longrightarrow \infty$ such that the geodesic segment from $p$ to $p_m$ tends to $\gamma$ as $m \rightarrow \infty$.
\label{cor_131}
\end{corollary}

\begin{proof} According to Proposition \ref{prop_1155}(ii), any geodesic passing through $p$
takes the form $\gamma_{p,v}$ for some $v \in S^{n-1}$, and hence it is a complete, minimizing geodesic.
Thus the cut-locus of $p$ is empty,
and consequently $\exp_{p}: T_p M \rightarrow M$ is a diffeomorphism onto $M$.
The ``Moreover'' part follows from Lemma \ref{lem_121} and Proposition \ref{prop_1155}, which imply that the geodesic ray $\gamma_{p, v}([0, \infty))$ is a limit of a sequence of minimizing geodesics
connecting $p$ with points in $L \setminus \{ p \}$ that tend to infinity.
\end{proof}

\begin{remark}{\rm The proof of Theorem \ref{thm_1217_} is quite robust, and in fact the assumption that $M$ is a Riemannian manifold in Theorem \ref{thm_1217_} can be weakened to the requirement that $M$
is a reversible Finsler manifold.
Moreover, we  think that for a suitable notion of a quasi-net,
the space $\RR^n$ in Theorem \ref{thm_1217_} may be replaced by other Hadamard manifolds.
For example, while Definition \ref{def_409} above
seems specific to the Euclidean space, it actually may be replaced by the ideal boundary of a
Hadamard manifold, see e.g. \cite{BG}.  }
\end{remark}

\medskip
For $v \in S^{n-1}$ define
\begin{equation}  B_v(x) = \inf_{B \in \partial_v M} B(x). \label{eq_341} \end{equation}
Recall that  a $1$-Lipschitz function $f$ foliates if $M$ is covered by transport lines of $f$.

\begin{proposition} For any $v \in S^{n-1}$, the function $B_v$ foliates and belongs to $\partial_v M$.
\label{lem_357}
\end{proposition}

\begin{proof} Since $\partial_v M \neq \emptyset$ by Lemma \ref{lem_1045}, the infimum in (\ref{eq_341}) is well-defined,
and $B_v(p) = \langle p, v \rangle$ for any $p \in L$. The function $B_v$ is $1$-Lipschitz, being the infimum of a
family of $1$-Lipschitz functions. Consequently,
$$ B_v \in \partial_v M. $$
By Lemma \ref{lem_121} there exists a sequence $L \ni p_m \rightsquigarrow v$ narrowly, that is,
\begin{equation}   \langle p_m, v \rangle - |p_m|  \xrightarrow{m \rightarrow \infty}  0 \qquad \text{and} \qquad |p_m| \xrightarrow{m \rightarrow \infty} \infty. \label{eq_151_}
\end{equation}
According to Lemma \ref{lem_1045} we may pass to a subsequence,
and assume that $d_{p_m} \longrightarrow B \in \partial_v M$. By Corollary \ref{cor_131} the cut-locus of $p_m$ is empty for all $m$.
Proposition \ref{lem_519}(i) thus implies that $B$ foliates. It remains to prove that $B \equiv B_v$. To this end we note
that since $B_v$ is $1$-Lipschitz, for any $x \in M$ and $m \geq 1$,
$$ B_v(x) \geq B_v(p_m) - d(p_m, x) = \langle p_m, v \rangle - d(p_m, x) = d_{p_m}(x) + (\langle p_m, v \rangle - |p_m|) \xrightarrow{m \to \infty} B(x),
$$
where we also used (\ref{eq_151_}) in the last passage. Thus $B_v \geq B$. However since $B \in \partial_v M$,
the inequality $B_v \leq B$ follows from the definition (\ref{eq_341}). Hence $B \equiv B_v$.
\end{proof}

\begin{lemma} Let $p,q \in L$ be two distinct points. Then for any $v \in S^{n-1}$,
\begin{equation} q \in \gamma_{p, v}( (0, \infty) ) \qquad \Longleftrightarrow \qquad v = \frac{q-p}{|q-p|}. \label{eq_504_} \end{equation}
\label{lem_906}
\end{lemma}

\begin{proof} We know that $\gamma_{p, v}$ is a transport line
of the $1$-Lipschitz function $B_v \in \partial_v M$, by Proposition \ref{prop_1155}(iii) and Proposition \ref{lem_357}.
Hence, from Corollary \ref{cor_230},
$$ q \in \gamma_{p, v}  \qquad \Longleftrightarrow \qquad |B_v(q) - B_v(p)| = d(p,q). $$
Set $w = (q - p) / |q-p|$.
Since $B_v(q) - B_v(p) = \langle q - p, v \rangle$ and $d(p,q) = |q-p|$, we conclude that
\begin{equation} q \in \gamma_{p, v}
 \qquad \Longleftrightarrow \qquad v = \pm w. \label{eq_513_}
\end{equation}
By Proposition \ref{prop_1155}(ii), we know that $\gamma_{p, -w}( (0, \infty) ) = \gamma_{p,w}( (-\infty, 0))$.
Since $B_w(q) > B_w(p)$ and since $\gamma_{p, w}$ is a transport line of $B_{w}$ with $\gamma_{p, w}(0) = p$,
we deduce from (\ref{eq_513_}) that $q \in \gamma_{p, w}( (0, \infty) )$ but $q \not \in \gamma_{p,w}( (-\infty, 0)) =\gamma_{p, -w}( (0, \infty) )$. Thus (\ref{eq_504_}) follows from (\ref{eq_513_}).
\end{proof}

The rest of this paper is devoted almost exclusively to the proof of Theorem \ref{thm_1217}.
From now on and until the end of Section \ref{sec_no_conj} we further assume that
the discrete set $L \subseteq \RR^n$ satisfies the quasi-net condition (QN2) in addition to (QN1), and that
$$ n = 2. $$
Thus $M$ is a two-dimensional manifold homeomorphic to $\RR^2$, by Corollary \ref{cor_131}.

\begin{lemma} For any $p \in L$, the odd map
$F(v) = \dot{\gamma}_{p,v}(0)$ from $S^1$ to $S_p M$ is a homeomorphism.
\label{lem_1008}
\end{lemma}

\begin{proof} The function $F$ is continuous and onto, by Proposition \ref{prop_1155}(ii). We  need to prove that $F$ is one-to-one.
We will use the following one-dimensional topological fact: If $h: S^{1} \rightarrow S^{1}$ is a continuous map, and $A, B \subseteq S^1$ are two dense subsets with $A = h^{-1}(B)$ such that the restriction of $h$ to $A$ is one-to-one, then the function $h: S^1 \rightarrow S^1$ is one-to-one. Write $$ A = \left \{ \frac{q - p}{|q - p|} \, ; \, q \in L \setminus \{ p \} \right \}, $$ which is a dense subset of $S^1$ by Lemma \ref{lem_121}.
By Lemma \ref{lem_906} and the ``Moreover'' part in Corollary \ref{cor_131}, the set  $B := F(A)$ is a dense subset of $S_p M$. For any $b \in B$, the geodesic ray emanating from $p$
in direction $b \in S_p M$ contains a point $q \in L \setminus \{ p \}$, and hence $F^{-1}(b) \subseteq S^1$ is a singleton  by
Lemma \ref{lem_906}.
Consequently, $A = F^{-1}(B)$ and the restriction of $F$ to $A$ is one-to-one.
In view of the above fact, $F$ is one-to-one.
\end{proof}

Let $\gamma: \RR \rightarrow M$ be any simple curve with $\lim_{t \rightarrow \pm \infty} \gamma(t) = \infty$
(for example, any complete minimizing geodesic has this property).
 The curve $\gamma$ induces a simple closed  curve in the one-point compactification of $M$ which is homeomorphic to the two-dimensional sphere.
From the Jordan curve theorem we learn that $M \setminus \gamma$ consists of two connected components, each of which is
homeomorphic to $\RR^2$
by the Sch\"onflies theorem.

\medskip When $\gamma_1, \gamma_2: \RR \rightarrow M$ are disjoint simple curves with
$\lim_{t \rightarrow \pm \infty} \gamma_i(t) = \infty$ for $i=1,2$, the set
$M \setminus (\gamma_1 \cup \gamma_2)$ thus consists of three connected components.
Exactly one of these three connected components is {\it in the middle}, in the sense that any curve connecting the two other components, has to intersect the middle one.
Denote the set of all points in this middle connected component by $\btwn(\gamma_1,\gamma_2)$, so as to say that a point $x\in \btwn(\gamma_1,\gamma_2)$ is {\it between} $\gamma_1$ and $\gamma_2$.

\begin{lemma} For any $x \in M$ and $v \in S^1$, there exist $p, q \in L$ such that  $x \in \btwn(\gamma_{p,v}, \gamma_{q,v})$. \label{lem_118}
\end{lemma}

\begin{proof} By Proposition  \ref{lem_357} there is a complete minimizing geodesic $\eta$ with $\eta(0) = x$
which is a transport line of $B_v$. By the Jordan curve theorem, the curve $\eta$ separates the manifold $M$ into two connected components.
The geodesic $\eta$ cannot contain the entire set $L$: Otherwise, it follows from (\ref{eq_600}) that all points of $L$
are contained in a single straight line in $\RR^n$, and this possibility is ruled out
by either (QN1) or (QN2). Hence there exists $p \in L \setminus \eta$. Write $H^+$ for the connected component of $M \setminus \eta$
that contains $p$, and write $H^-$ for the other connected component.

\medskip By Corollary \ref{cor_131}, there exists a minimizing geodesic ray $\gamma$ emanating from $p$ that passes through $x$.
This geodesic ray crosses the geodesic $\eta$ at the point $x$. By Corollary \ref{cor_131}, we may find a sequence
$q_m \in L$ such that $q_m \longrightarrow \infty$ and such that the geodesic segment from $p$ to $q_m$ tends to $\gamma$.
Thus for a sufficiently large $m$, the geodesic segment from $p$ to $q_m$ crosses $\eta$ at a point close to $x$.
This minimizing geodesic segment cannot cross $\eta$ twice, since $\eta$ is a complete, minimizing geodesic. Hence $q_m \in H^-$.
Fix such  $m$ and set $q: = q_m \in L \cap H^-$.

\medskip We have thus found points $p \in L \cap H^+$ and $q \in L \cap H^-$. From Proposition \ref{prop_1155}(iii), the complete minimizing geodesics  $\gamma_{p,v}$ and $\gamma_{q,v}$ are transport lines of $B_v \in \partial_v M$. By Lemma \ref{lem_1005} they cannot intersect $\eta$, thus $\gamma_{p,v} \subset H^+$ while $\gamma_{q,v} \subset H^-$. Hence the entire curve $\eta$ lies between $\gamma_{p,v}$ and $\gamma_{q,v}$, and in particular $x = \eta(0) \in \btwn(\gamma_{p,v}, \gamma_{q,v})$.
\end{proof}

\begin{lemma}
Let $p \in L$, $v \in S^1$ and write $H^+$ and $H^-$ for the two connected components of $M \setminus \gamma_{p,v}$.
Then  there exists a unit vector $v^{\perp} \in S^1$ orthogonal
to $v$ such that for any $q \in L$,
\begin{equation} \left \langle q-p, \pm v^{\perp} \right \rangle > 0 \qquad \Longrightarrow \qquad q \in H^\pm. \label{eq_710}
\end{equation}
\label{lem_0158}
\end{lemma}

\begin{proof} Corollary \ref{cor_131} implies that two distinct geodesic rays
 emanating from $p$ are disjoint except for their intersection at $p$.
 Therefore any geodesic ray from $p$ is contained either in $H^+$ or in $H^-$ or in $\gamma_{p,v}$. Set
\begin{equation}  I^{\pm} = \left \{ w \in S^1 \, ; \, \gamma_{p, w}( (0, \infty) ) \subset H^{\pm} \right \}. \label{eq_1020} \end{equation}
It follows from Lemma \ref{lem_1008} that $I^+ \cup I^- = S^1 \setminus \{v, -v \}$ with $I^+ \cap I^- = \emptyset$.
Since $I^{\pm}$ is an open subset of $S^1$, we conclude that $I^{\pm}$ is an open arc in $S^1$ that stretches from the point $v$ to its antipodal point $-v$. Write $v^{\perp} \in S^1$ for the unique
vector in $I^+$ that is orthogonal to $v$.

\medskip Let $q \in L$ satisfy $\left \langle q-p, \pm v^{\perp} \right \rangle > 0$. Then $q \neq p$,
and the unit vector $w = (q - p) / |q-p|$ belongs to the arc $I^{\pm}$. Lemma \ref{lem_906} shows
that $q \in \gamma_{p,w}((0, \infty))$ while $\gamma_{p,w}((0, \infty)) \subset H^{\pm}$ by (\ref{eq_1020}). Hence $q \in H^{\pm}$.
\end{proof}

Our knowledge regarding geodesics passing through $L$-points will be applied in order to deduce certain large-scale properties of the manifold $M$. Proposition \ref{lem_net} shows that  the distances of faraway points from $L$ cannot grow too fast. The following Proposition \ref{lem_445} implies that the large-scale geometry of $M$ tends to the Euclidean one.

\begin{proposition}
Let $p \in L$. Then for any $v \in S^1$,
\begin{equation}  \lim_{t \rightarrow \infty} \frac{d( \gamma_{p, v}(t), L )}{t} = 0. \label{eq_1116} \end{equation}
Moreover, the convergence is uniform in $v \in S^1$.
\label{lem_net}
\end{proposition}

\begin{proof} Fix $v \in S^1$ and $0 < \eps < 1$. We will show that there exists $T$  such that for any $t > T$,
\begin{equation} d( \gamma_{p,v}(t), L ) \leq |p| + 21 \sqrt{\eps} \cdot t. \label{eq_654} \end{equation}
Abbreviate $\gamma = \gamma_{p,v}$. In view of Lemma \ref{lem_0158}, we may write $H^+$ and $H^-$ for the two connected components of $M \setminus \gamma$,
and conclude the existence of a unit vector $$ v^{\perp} \in S^1 $$ orthogonal
to $v$ such that (\ref{eq_710}) holds true for any $q \in L$.
Let us apply condition (QN2). It implies that there exist two sequences $(p_m^+)_{m \geq 1}$ and $(p_m^-)_{m \geq 1}$ in the discrete set $L$
such that for all $m \geq 1$,
\begin{equation}  \left \langle \frac{p_m^\pm}{|p_m^\pm|}, \pm v^{\perp} \right \rangle > \eps \qquad \text{and} \qquad \left \langle \frac{p_m^\pm}{|p_m^\pm|}, v \right \rangle
> 1 - \eps  \label{eq_749} \end{equation}
and such that $p_m^\pm \longrightarrow \infty$ while  $|p^{\pm}_{m+1}| / |p_m^{\pm}| \longrightarrow 1$. See an illustration in Figure \ref{fig1}. From (\ref{eq_710}) and (\ref{eq_749}) we conclude that $p_m^{\pm} \in H^{\pm}$ whenever $|p_m^{\pm}| > |p| / \eps$. Hence there exists $M \geq 1$ such that for all $m \geq M$,
\begin{equation}
p_m^\pm \in H^\pm \qquad \text{and} \qquad \frac{|p^{\pm}_{m+1}|}{|p_m^{\pm}|} < 1 + \eps.
\label{eq_928} \end{equation}
Set $T = \max_{1 \leq m \leq M} \max \{ |p_m^{+}|, |p_m^-| \}$. For $t > T$, define
$$ m^{\pm} = m^{\pm}(t) := \min \{ m \geq 1 \, ;\,  |p_m^{\pm}| > t \}. $$
The positive integers $m^{\pm} > M$ are well-defined as $p_m^{\pm} \longrightarrow \infty$. According to (\ref{eq_928}),
\begin{equation} a := p^+_{m^+} \in H^+ \qquad \text{and} \qquad b := p^-_{m^-} \in H^-. \label{eq_1027} \end{equation}
Since $|p_{m^{\pm}-1}^{\pm}| \leq t$, by (\ref{eq_928}) we also have
\begin{equation}
1 < \frac{|a|}{t} < 1 + \eps
\qquad \text{and} \qquad
1 < \frac{|b|}{t} < 1 + \eps.
\label{eq_942}
\end{equation}
From (\ref{eq_749}) we know that $\langle a / |a|, v \rangle > 1 - \eps$
and $\langle b / |b|, v \rangle > 1 - \eps$. Hence the Euclidean distance between $a / |a|$ and $b / |b|$ is less than $4 \sqrt{\eps}$.
From (\ref{eq_942}) we conclude that
\begin{equation}
 |a - b| \leq \left| a - \frac{|a|}{|b|}b\right| + \left| \frac{|a|}{|b|}b - b \right| \leq (1 + \eps) t \left| \frac{a}{|a|} - \frac{b}{|b|} \right|
+ (1 + \eps) t  \left| \frac{|a|}{|b|} - 1 \right| \leq 10 \sqrt{\eps} t.
\label{eq_1031_}
\end{equation}
\begin{figure}
\begin{center} \includegraphics[width=4in]{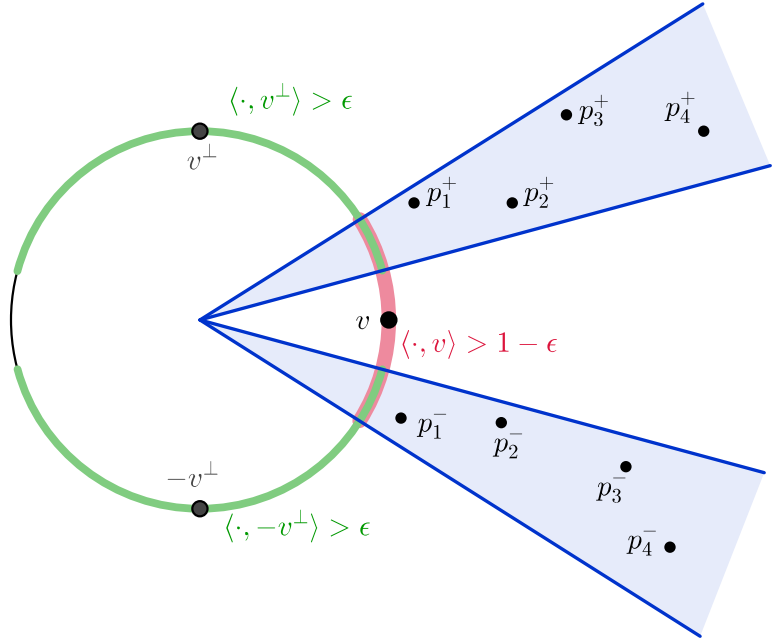} \end{center}
\caption{The $L$-points used in the proof of Proposition \ref{lem_net} \label{fig1}}
\end{figure}
Since $B_v(a) = \langle a, v \rangle$ while $\langle a / |a|, v \rangle > 1 - \eps$ we may use (\ref{eq_942}) and bound $B_v(a)$ as follows:
\begin{equation}  (1 + \eps) t \geq |a| \geq B_v(a) = |a| \langle a/ |a|, v \rangle \geq |a| (1 - \eps) \geq (1 - \eps) t. \label{eq_1022_} \end{equation}
From Proposition 3.7(iii) we know that $\gamma = \gamma_{p,v}$ is a transport line of $B_v$ with $\gamma(0) = p$.
Hence $B_v(\gamma(t)) = t + B_v(p) \in (t - |p|, t + |p|)$. Thus, from (\ref{eq_1022_}),
\begin{equation}
|B_v(a) - B_v(\gamma(t))| \leq |p| + \eps t. \label{eq_1023}
\end{equation}
It follows from (\ref{eq_1027}) that the minimizing geodesic
between $a$ and $b$ intersects the curve $\gamma$ at a point $\gamma(t_1)$ for some $t_1 \in \RR$. From (\ref{eq_1031_}),
\begin{equation}
d(\gamma(t_1), a) \leq d(a,b) = |a-b| \leq 10 \sqrt{\eps} t. \label{eq_1035}
\end{equation}
Furthermore, from (\ref{eq_1023}) and (\ref{eq_1035}),
\begin{align} |t_1 - t|  = |B_v(\gamma(t_1)) - B_v(\gamma(t))| \leq d(\gamma(t_1), a) + |B_v(a) - B_v(\gamma(t))|
 \leq 11 \sqrt{\eps} t + |p|. \label{eq_533}
\end{align}
By (\ref{eq_1035}) and (\ref{eq_533}), since $a \in L$,
$$
d(\gamma(t), L) \leq |t_1 - t| + d(\gamma(t_1), L) \leq |t_1 - t| + d(\gamma(t_1), a)
\leq |p| + 21 \sqrt{\eps} t, $$
as advertised in (\ref{eq_654}). This implies the convergence in (\ref{eq_1116}). In order to conclude that the convergence is uniform
in $v \in S^1$,
we will argue that there exists $\bar{T} > 0$ that does not depend on $v$,
such that (\ref{eq_654}) holds true for all $t > \bar{T}$ and $v \in S^1$.
To this end, we fix finitely many non-empty open sectors $S_1,\ldots,S_R \subseteq \RR^2$, such that
any set of the form
$$ A_{v, \pm } = \left \{ 0 \neq x \in \RR^2 \, ; \, \left \langle \frac{x}{|x|}, \pm v^{\perp} \right \rangle > \eps,  \left \langle  \frac{x}{|x|}, v \right \rangle > 1 - \eps \right \} \qquad \qquad (v \in S^1) $$
 contains one of these sectors. From (QN2), in each sector $S_i$ there exists
a sequence of points $p_m^{(i)} \in L \cap S_i \ (m=1,2,\ldots)$ such that
$p_m^{(i)} \longrightarrow \infty$ and $|p_{m+1}^{(i)}| / |p_m^{(i)}| \longrightarrow 1$ as $m$ tends to infinity.

\medskip Consequently, for any $v \in S^1$ and any choice of sign $\pm$ there exists $i=1,\ldots,R$ such that $p_m^{\pm} := p_m^{(i)}$
satisfy (\ref{eq_749}) while $p_m^{\pm} \longrightarrow \infty$ and $|p_{m+1}^{\pm}| / |p_m^{\pm}| \longrightarrow 1$.
We conclude that the parameter $M$ above can be chosen independent of $v$, and the parameter $T$ above
is upper bounded by some $\bar{T}$ which does not depend on $v$. This implies that the convergence in
(\ref{eq_1116}) is uniform in $v$.
\end{proof}

\begin{proposition} Let $p \in L$. Then for $x, y \in \RR^2$, writing $x = a v, y = bw $ with $v, w \in S^1$ and $a, b \geq 0$, we have
$$  \lim_{r \rightarrow \infty} \frac{d(\gamma_{p,v}( a r), \gamma_{p,w}( b r))}{r} = |x - y|, $$
and the convergence is locally uniform in $x, y \in \RR^2$.
\label{lem_445}
\end{proposition}

\begin{proof} Abbreviate $\gamma_v = \gamma_{p,v}$. In order to prove the local uniform convergence, we
 fix no less than five sequences
$$ r_m \longrightarrow \infty, \ \   S^1 \ni w_m \longrightarrow w, \ \     S^1 \ni v_m \longrightarrow v, \ \  [0, \infty) \ni a_m \longrightarrow a, \
\ [0, \infty) \ni b_m \longrightarrow b. $$
Since $|x-y|$ is a continuous function of $x,y \in \RR^2$, the local uniform convergence would follow
once we show that
\begin{equation}  \lim_{m \rightarrow \infty} \frac{d(\gamma_{v_m}( a_m r_m), \gamma_{w_m}( b_m r_m))}{r_m} = |x - y|. \label{eq_359} \end{equation}
Passing to a subsequence, we may assume that the limit in (\ref{eq_359}), denoted by $L$, is attained as an element of $\RR \cup \{ + \infty \}$,
and our goal is to prove that $L = |x - y|$. Passing to a further subsequence, we may assume that $\lim_{m} a_m r_m$
and $\lim_{m} b_m r_m$ exist as elements in $\RR \cup \{ + \infty \}$.
First consider  the case where $\lim_m a_m r_m < \infty$. In this case
necessarily $a_m \longrightarrow 0 = a$, hence $x = 0$ and
$$
L = \lim_{m \rightarrow \infty} \frac{d(\gamma_{v_m}( a_m r_m), \gamma_{w_m}( b_m r_m))}{r_m} = \lim_{m \rightarrow \infty} \frac{d(p, \gamma_{w_m}( b_m r_m))}{r_m} = \lim_{m \rightarrow \infty} \frac{ b_m r_m}{r_m}
= b = |x-y|. $$
Therefore $L = |x-y|$  when $\sup_m a_m r_m < \infty$. Similarly, $L = |x-y|$  when $\sup_m b_m r_m < \infty$. We may thus assume that
\begin{equation}
\lim_{m \rightarrow \infty} a_m r_m = \lim_{m \rightarrow \infty} b_m r_m = +\infty. \label{eq_602_}
\end{equation}
According to Proposition \ref{lem_net}, for any $m$ we may select $p_m,q_m \in L$ with
\begin{equation}  \lim_{m \rightarrow \infty} \frac{d(p_m,  \gamma_{v_m}(a_m r_m))}{a_m r_m} = 0
\qquad \text{and} \qquad  \lim_{m \rightarrow \infty} \frac{d(q_m,  \gamma_{w_m}(b_m r_m))}{b_m r_m} = 0. \label{eq_816_}
\end{equation}
By the triangle inequality, we know that
$$
\left| \frac{|p_m-p|}{a_{m}r_{m}} - 1 \right|\leq
\frac{d(p_m, \gamma_{v_m}(a_m r_m))}{a_{m}r_{m}}
\qquad \text{and} \qquad
\left| \frac{|p_m|}{a_{m}r_{m}} - \frac{|p_m-p|}{a_{m}r_{m}} \right| \leq \frac{|p|}{a_{m}r_{m}}.
$$
Using (\ref{eq_602_}) and (\ref{eq_816_}), we have that
$|p_m| / (a_m r_m) \longrightarrow 1$
and similarly $|q_m| / (b_m r_m) \longrightarrow 1$.
Moreover, since $B_{v_m} \in \partial_{v_m}M$ by Proposition \ref{lem_357}, and $\gamma_{v_m}$ is its transport line by Proposition \ref{prop_1155}(iii),
\begin{equation}  \lim_{m \rightarrow \infty} \left \langle \frac{p_m}{|p_m|}, v_m \right \rangle =
\lim_{m \rightarrow \infty} \frac{B_{v_m}(p_m)}{|p_m|} =
\lim_{m \rightarrow \infty} \frac{B_{v_m}( \gamma_{v_m}(a_m r_m) )}{a_m r_m} = 1. \label{eq_425} \end{equation}
Since $S^1 \ni v_m \longrightarrow v$, we learn from (\ref{eq_425}) that $p_m / |p_m| \longrightarrow v$.
Similarly $q_m / |q_m| \longrightarrow w$.
Consequently,
$$
L =  \lim_{m \rightarrow \infty} \frac{|p_m - q_m|}{r_m} =
\lim_{m \rightarrow \infty} \left| a_m \frac{|p_m|}{a_m r_m}  \frac{p_m}{|p_m|} -  b_m \frac{|q_m|}{b_m r_m} \frac{q_m}{|q_m|}  \right|
= |a v - b w| = |x-y|.
$$
\end{proof}

At this point in the proof we quote the area growth theorem of Bangert and Emmerich \cite{BE2}.
We write $D_M(x, r) = \{ y \in M \, ; \, d(x,y) < r \}$, the geodesic ball of radius $r$ centered at $x$.
Theorem 1 from \cite{BE2} reads as follows:

\begin{theorem}[Bangert and Emmerich] Let $M$ be a complete, two-dimensional Riemannian manifold
without conjugate points, diffeomorphic to $\RR^2$ and let $x \in M$. Then,
$$ \liminf_{r \rightarrow \infty} \frac{\area(D_M(x, r))}{\pi r^2} \geq 1, $$
with equality if and only if $M$ is flat.
\label{thm_BE}
\end{theorem}

In order to prove Theorem \ref{thm_1217} it thus remains to show that $M$ has no conjugate points,
and to use the fact that the large-scale geometry of $M$ is approximately Euclidean
in order to prove that $\area(D_M(x,r)) / r^2 \longrightarrow \pi$ as $r$ tends to infinity. We remark that
there are surfaces whose large-scale geometry is approximately Euclidean, such as a compactly-supported perturbation of
the Euclidean plane, yet they have conjugate points and  consequently they are not isometric to the Euclidean plane.

\section{The ideal boundary}

Our goal in this section is to prove that $\partial_v M$ is a singleton for each $v \in S^1$.
We begin with the following:

\begin{lemma} Let $q \in L, v,w \in S^1$ and $C \in \RR$.
 Assume that for any $t > 0$,
\begin{equation}   B_v(\gamma_{q,w}(t)) \geq t + C. \label{eq_313} \end{equation}
Then $v = w$. \label{lem_307}
\end{lemma}

\begin{proof} The geodesic $\gamma = \gamma_{q,w}$ is a transport line of $B_w$, by Proposition \ref{prop_1155}(iii).
It thus follows from (\ref{eq_313}) that there exists $C' \in \RR$ such that
\begin{equation}  \min \{ B_v(\gamma(t)), B_w(\gamma(t)) \} \geq t + C' \qquad \qquad \qquad (t > 0). \label{eq_317} \end{equation}
According to  Proposition \ref{lem_net}, for any $t > 0$ there exists $p_t \in L$ with
\begin{equation}  \lim_{t \rightarrow \infty} \frac{d(\gamma(t), p_t)}{t} = 0. \label{eq_1111} \end{equation}
Since $d(\gamma(t), \gamma(0)) = t$, we have that $| d(p_t, 0) - t| = o(t)$ as $t \rightarrow \infty$, by the triangle inequality.
Moreover, from (\ref{eq_317}), (\ref{eq_1111}), and the fact that $B_v \in \partial_v M$ and $B_w \in \partial_w M$ are  $1$-Lipschitz,
\begin{align} \nonumber \min \{ \langle p_t, v \rangle, \langle p_t, w \rangle \} & =
\min \{ B_v(p_t), B_w(p_t) \} = \min \{ B_v(\gamma(t)), B_w(\gamma(t)) \} + o(t) \\ & \geq t + o(t) = d(0,p_t) + o(t) = |p_t| + o(t).
\label{eq_320}
\end{align}
Since $|p_t| = t + o(t)$, we know that $|p_t|$ tends to infinity with $t$,
and we deduce from (\ref{eq_320}) that
\begin{equation}  \lim_{t \rightarrow \infty} \min \left \{  \left \langle \frac{ p_t}{|p_t|} , v \right \rangle, \left \langle \frac{ p_t}{|p_t|} , w \right \rangle \right \} = 1. \label{eq_1008} \end{equation}
Since $|v| = |w| = 1$, it follows from (\ref{eq_1008}) that $p_t / |p_t| \longrightarrow v$
and $p_t / |p_t| \longrightarrow w$. Hence $ v = w$.
\end{proof}

Recall that $B_{\gamma_{p,v}}$ is the Busemann function of the geodesic $\gamma_{p,v}$.
The following proposition is a step in the proof that $\partial_v M$ is a singleton.
Its  proof  requires several lemmas.

\begin{proposition} For any $q \in L, v \in S^1$ and $x \in M$ we have $B_{\gamma_{q,v}}(x) = B_v(x) - \langle q, v \rangle$.
\label{prop_322}
\end{proposition}

\begin{lemma}
Let $f:M\to\mathbb{R}$ be a $1$-Lipschitz function and let $\gamma:\mathbb{R}\to M$ be a transport line of	$f$. Then for any $x\in M$,
$$ B_{\gamma}(x)+f(\gamma(0))\leq f(x). $$
\label{lem_459}
\end{lemma}

\begin{proof}
Since $\gamma$ is a transport line of $f$, we have $f(\gamma(t)) = t + f(\gamma(0))$ for all $t \in \RR$.
Since $f$ is $1$-Lipschitz, for any $x \in M$,
\begin{equation*}  B_{\gamma}(x) = \lim_{t \rightarrow \infty} [t - d(x, \gamma(t))] \leq \lim_{t \rightarrow \infty} [t - (f(\gamma(t)) - f(x))]
= f(x) - f(\gamma(0)).  \end{equation*}
\end{proof}

\begin{lemma} Let $p,q \in L$ and $v \in S^1$.
Suppose that for every $t > 0$ there is a point $q_t \in L \setminus \{ q \}$ with $|q_t| > t$ such that the geodesic segment  from $q$ to $q_t$ passes through the ball $D_M(\gamma_{p,v}(t), 1/t)$. Then $q_t \rightsquigarrow v$ (i.e., $q_t / |q_t| \longrightarrow v$ as $t \to \infty$).
\label{lem_250}
\end{lemma}

\begin{proof} Let $(t_m)_{m \geq 1}$ be an increasing sequence tending to infinity with $q_{t_m} / |q_{t_m}| \longrightarrow w \in S^{1}$. Our goal
is to prove that $w = v$.  With a slight abuse of notation we abbreviate
$$ q_m = q_{t_m}. $$
 Write $\eta_m: \RR \rightarrow M$ for the minimizing geodesic
 with $\eta_m(0) = q$ and $\eta_m(d(q,q_m)) = q_m$, uniquely determined by Corollary \ref{cor_131}.
Since $q_m \rightsquigarrow w$, it follows from Proposition \ref{prop_1155}(i) that
\begin{equation}  \eta_m \xrightarrow{m \to \infty}  \gamma_{q, w}. \label{eq_309} \end{equation}
By our assumptions, for any $m$ there exists a point $\eta_m(s_m)$  on the geodesic
segment between $q$ and $q_m$ with \begin{equation}  \eta_m(s_m) \in D_M(\gamma_{p,v}(t_m), 1 / t_m). \label{eq_403} \end{equation}
Since $d(\gamma_{p,v}(t_m), p) = t_m$ and $d(\eta_m(s_m), q) = s_m$, from (\ref{eq_403}) and the triangle inequality,
\begin{equation}  |s_m - t_m| < d(p,q) + 1/ t_m.  \label{eq_307} \end{equation}
From Proposition \ref{prop_1155}(iii), there exists $C \in \RR$ such that $B_v(\gamma_{p,v}(t)) = t + C$ for all $t \in \RR$.
Since $B_v$ is $1$-Lipschitz, from (\ref{eq_403}) and (\ref{eq_307}),
\begin{equation}  B_v(\eta_m(s_m)) \geq B_v(\gamma_{p,v}(t_m)) - \frac{1}{t_m} = t_m + C - \frac{1}{t_m} \geq s_m + C', \label{eq_303} \end{equation}
with $C' = C  - 2/ t_1 - d(p,q)$. Since $t_m \longrightarrow \infty$, we learn from (\ref{eq_307}) that $s_m
\longrightarrow \infty$ as well. Thus, for any fixed $s > 0$, there exists $m$
with $s_m > s$ and according to (\ref{eq_303}),
\begin{equation}
B_v(\eta_m(s)) - s \geq B_v(\eta_m(s_m)) - s_m \geq C'. \label{eq_308} \end{equation}
By letting $m$ tend to infinity we see from (\ref{eq_309}) and (\ref{eq_308}) that for all $s > 0$,
$$  B_v(\gamma_{q, w}(s)) \geq s + C'. $$
Lemma \ref{lem_307} now shows that $v = w$.
\end{proof}

\begin{lemma} Let $p,q \in L$ and $v \in S^{1}$. Then,
\begin{equation}  \lim_{t \rightarrow \infty} [t - d(\gamma_{q, v}(t), p)] = \langle p - q, v \rangle. \label{eq_143} \end{equation}
\label{lem_903}
\end{lemma}

\begin{proof} Abbreviate $\gamma = \gamma_{q,v}$ and let $T > 0$ be such that $\gamma(t) \neq p$
for $t > T$.
For $t > T$, the geodesic ray emanating from $p$ (or from $q$) that passes through $\gamma(t)$ may be approximated arbitrarily well by a geodesic ray from $p$ (or from $q$)
that passes through a faraway point of $L$, according to
Corollary \ref{cor_131}.
Thus there exist $$ p_{t} \in L \setminus \{p \} \qquad \text{and} \qquad q_t \in L \setminus \{ q \} $$
with $\min \{ |p_{t}|, |q_t| \} > t$ such that the following property holds:
 The geodesic segment from $p$ to $p_{t}$ passes through a point $y_t \in D_M(\gamma(t), 1/t)$,
 while the geodesic segment from $q$ to $q_t$ passes through a point $z_t \in D_M(\gamma(t), 1/t)$.
 Lemma \ref{lem_250} thus implies that
\begin{equation}  \lim_{t \rightarrow \infty} \frac{p_{t}}{|p_{t}|} = \lim_{t \rightarrow \infty} \frac{q_{t}}{|q_{t}|} = v.
\label{eq_338} \end{equation}
From the triangle inequality,
\begin{equation} d(p_{t}, q) - d(p_{t}, p) = d(p_{t}, q) - d(p_{t}, y_{t}) - d(y_{t}, p) \leq d(y_{t}, q) - d(y_{t}, p),
\label{eq_1134} \end{equation}
and
\begin{equation} d(q_{t}, q) - d(q_{t}, p) = d(q_{t}, z_{t}) + d(z_{t}, q) - d(q_{t},  p) \geq d(z_{t}, q) - d(z_{t}, p).
\label{eq_1134_} \end{equation}
Since $y_{t}, z_{t} \in D_M(\gamma(t), 1/t)$ we obtain from (\ref{eq_1134}) and (\ref{eq_1134_}) that
 for all $t > T$,
\begin{equation}  |p_{t} - q| - |p_{t} - p| - \frac{2}{t} \leq d(\gamma(t), q) - d(\gamma(t), p) \leq
|q_{t} - q| - |q_{t} - p| + \frac{2}{t}.  \label{eq_1146}
\end{equation}
However, from (\ref{eq_338}) we deduce that  both the left-hand side and the right-hand side of (\ref{eq_1146}) tend
to $\langle p - q, v \rangle$ as $t \rightarrow \infty$. Since $d(\gamma(t), q) = t$  we obtain
(\ref{eq_143}) by letting $t \rightarrow \infty$ in (\ref{eq_1146}).
\end{proof}

\begin{proof}[Proof of Proposition \ref{prop_322}] The Busemann function $\tilde{B} := B_{\gamma_{q,v}}$ is a $1$-Lipschitz
function. By Proposition \ref{prop_1155}(iii) we know that $\gamma_{q,v}$ is a transport line of $B_v$.
Hence, by Lemma \ref{lem_459} for $f = B_v$,
\begin{equation}  \tilde{B}(x) + \langle q, v \rangle = \tilde{B}(x) + B_{v}(q) \leq B_v(x) \qquad \qquad \qquad \text{for all} \ x \in M. \label{eq_502} \end{equation}
However, from Lemma \ref{lem_903}, for any $p \in L$,
\begin{equation}  \tilde{B}(p) = \lim_{t \rightarrow \infty} [t - d(\gamma_{q, v}(t), p)] = \langle p - q, v \rangle. \label{eq_1154_} \end{equation}
We conclude from (\ref{eq_1154_}) that the $1$-Lipschitz function
$$ x \mapsto \tilde{B}(x) + \langle q, v \rangle \qquad \qquad \qquad (x \in M) $$ belongs to $\partial_v M$.
This function is bounded from above by $B_v$, according to (\ref{eq_502}). From the definition (\ref{eq_341}) of $B_v$ as the smallest
element in $\partial_v M$, we conclude that $ B_v \equiv \tilde{B} + \langle q, v \rangle$.
\end{proof}

It follows from (\ref{eq_639}) and from the fact that $B_v$ is the minimal element in $\partial_v M$
that the {\it maximal} element in $\partial_{v} M$ satisfies
\begin{equation}
B^v(x) := \sup_{B \in \partial_v M} B(x) = -\inf_{B \in \partial_{-v} M} B(x) = -B_{-v}(x) \qquad \qquad (x \in M).
\label{eq_422}
\end{equation}
Since $B_{-v} \in \partial_{-v} M$ by Proposition \ref{lem_357}, we learn from (\ref{eq_639}) and (\ref{eq_422}) that
indeed
$$ B^v \in -\partial_{-v} M = \partial_v M. $$
For $v \in S^1$ we denote
$$ f_v = B^v - B_v. $$
The function $f_v: M \rightarrow \RR$ is clearly non-negative.
By (\ref{eq_422}) it is also evident that $f_v=f_{-v}$.

\begin{lemma} Let $p \in L, v \in S^1, \eps > 0$ and let $x \in M$ satisfy $f_v(x) < \eps$. Then there exists $t_0 > 0$ with the following property:
For any $t > t_0$ and for any point $y \in M$ lying on a minimizing geodesic segment connecting $x$ and $\gamma_{p, v}(t)$, we have
$$ f_v(y) < \eps. $$
\label{lem_101}
\end{lemma}

\begin{proof} Abbreviate $\gamma = \gamma_{p, v}$, and for $s > 0$ and $x \in M$ denote
$$ \delta(s,x) = d(\gamma(s), x) + d(\gamma(-s), x) - 2 s, $$
the deficit in the triangle inequality. 
The function $\delta(s,x)$ is non-increasing in $s$. Moreover,
since $ \gamma_{p, -v}(s) = \gamma_{p, v}(-s)$ by Proposition \ref{prop_1155}(ii), we deduce that for any $x \in M$,
$$ \lim_{s \rightarrow \infty} \delta(s,x) = -B_{\gamma_{p, v}}(x) -B_{\gamma_{p, -v}}(x) = -B_v(x) - B_{-v}(x) = B^v(x) - B_v(x) = f_v(x), $$
where we used Proposition \ref{prop_322} and (\ref{eq_422}) in the last passages.
Since $\delta(t,x) \searrow f_v(x) < \eps$ as $t \rightarrow \infty$,
there exists $t_0$ such that for any $t > t_0$,
\begin{equation}  \delta(t,x) < \eps. \label{eq_839} \end{equation}
Fix $t > t_0$. Since the point $y$ lies on a minimizing geodesic connecting $x$ and $\gamma(t)$,
it follows from (\ref{eq_839}) and the triangle inequality that
\begin{align*} \delta(t, y) & = d(\gamma(t), y) + d(\gamma(-t), y) - 2 t = d(\gamma(t), x) - d(x,y) + d(\gamma(-t), y) - 2t
\\ & \leq d(\gamma(t), x) + d(\gamma(-t), x) - 2t = \delta(t,x) < \eps.
\end{align*}
However, $\delta(s,y)$ is non-increasing in $s$. Therefore,
\begin{equation*} f_v(y) = \lim_{s \rightarrow \infty} \delta(s, y) \leq \delta(t,y) < \eps.  \end{equation*}
\end{proof}

In the proof of the following proposition we rely  on the fact that any geodesic through a point in $L$
is minimizing, according to Corollary \ref{cor_131}.

\begin{proposition} For any $v \in S^1$ we have $B^v \equiv B_v$ and hence $\partial_v M$ is a singleton.
\label{prop_313}
\end{proposition}

\begin{proof} Let $x \in M$ and $\eps > 0$. Our goal is to prove that $f_v(x) = B^v(x) - B_v(x) < \eps$.
To this end
we use Lemma \ref{lem_118}, according to which there exist $p, q \in L$ such that
\begin{equation}  x \in S:= \btwn(\gamma_{p, v}, \gamma_{q,v}). \label{eq_918} \end{equation}
We claim that:
\addtocounter{equation}{1}
\newcounter{eq_1235}
\setcounter{eq_1235}{\value{equation}}
\begin{enumerate}
\item[(\arabic{eq_1235})]
The geodesic from $q$ to any point in $\gamma_{p,v}$ is pointing into $S$
at the point $q$.
\addtocounter{equation}{1}
\newcounter{eq_1253}
\setcounter{eq_1253}{\value{equation}}
\item[(\arabic{eq_1253})] The geodesic segment from $q$ to $x$ is pointing into  $S$ at the point $q$, and it does not intersect $\gamma_{p,v}$.
\end{enumerate}
Indeed, by Corollary \ref{cor_131}, any complete geodesic through $q$
which is not $\gamma_{q, v}$ cannot intersect $\gamma_{q, v} \setminus \{ q \}$,
and (\arabic{eq_1235}) follows. We learn from (\ref{eq_918}) that the geodesic from $q$ to $x$
 cannot intersect $\gamma_{q,v} \setminus \{ q \}$, hence it is pointing into $S$ at the point $q$. Moreover, this geodesic cannot cross
 $\gamma_{p,v}$ twice, and it ends at a point in $S$, and therefore this geodesic segment cannot intersect $\gamma_{p,v}$ at all.
 This proves (\arabic{eq_1253}).

\medskip Proposition \ref{prop_322} tells us that $B_v = B_{\gamma_{p,v}} + \langle p,v \rangle$ and that $B_{-v} = B_{\gamma_{p, -v}} - \langle p, v \rangle$. Since $\gamma_{p,-v}(t) = \gamma_{p,v}(-t)$, this means that for any $y \in M$,
\begin{equation} B_v(y) = \lim_{t \rightarrow \infty} [C_t - d(y, \gamma_{p,v}(t))] \qquad \text{and} \qquad
B_{-v}(y) = \lim_{t \rightarrow \infty} [C_t' - d(y, \gamma_{p,v}(-t))] \label{eq_1236_} \end{equation}
for $C_t = t + \langle p,v \rangle$ and $C_t' = t - \langle p,v \rangle$. Recall that $B_v$ and $B_{-v}$ foliate by Proposition \ref{lem_357}, and that
$\nabla B_v(q) = \dot{\gamma}_{q,v}(0) = -\dot{\gamma}_{q,-v}(0) = -\nabla B_{-v}(q)$ thanks to Proposition \ref{prop_1155}. In view of (\ref{eq_1236_}) we may invoke Proposition \ref{lem_519}(ii)
and conclude the following: The minimizing geodesic from $q$ to $\gamma_{p,v}(t)$ tends to $\gamma_{q,v}$ as $t \rightarrow \infty$,
and the minimizing geodesic from $q$ to $\gamma_{p, v}(-t)$ tends to $\gamma_{q, -v}$ as $t \rightarrow \infty$.

\medskip In other words, the angle with $\gamma_{q,v}$ of the geodesic from $q$ to $\gamma_{p,v}(t)$ tends to zero as $t \rightarrow \infty$,
and the angle with $\gamma_{q,v}$ of the geodesic from $q$ to $\gamma_{p,v}(-t)$ tends to $\pi$ as $t \rightarrow \infty$.

\medskip {\it Claim:} There exists $t_0 \in \RR$ such that $x$ lies on the geodesic segment from $q$ to $\gamma_{p, v}(t_0)$.

\medskip Indeed, by continuity of the angle, for any given angle $\alpha \in (0, \pi)$ there exists $t \in \RR$ such that
the angle with $\gamma_{q,v}$ of the geodesic from $q$ to $\gamma_{p,v}(t)$ equals $\alpha$.
We conclude
that for any unit vector $u \in S_q M$ that is pointing into  $S$, there exists $y \in \gamma_{p,v}$ such that the geodesic from $q$
to $y$ is tangent to $u \in S_q M$.

\medskip We know from (\arabic{eq_1253}) that the geodesic from $q$ to $x$ is pointing into  $S$ at the point $q$.
We thus learn from the previous paragraph that there exists
$y \in \gamma_{p,v}$ such that the geodesic segment from $q$ to $x$ coincides near $q$ with the geodesic segment from $q$ to $y$. By writing
$y  = \gamma_{p, v}(t_0)$, the claim follows from (\arabic{eq_1253}).

\begin{figure}
\begin{center} \includegraphics[width=4.8in]{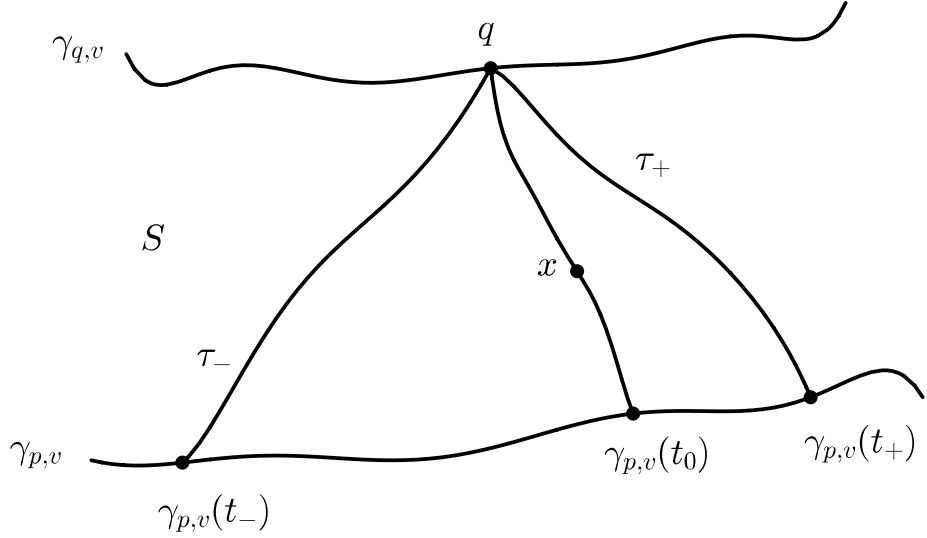} \end{center}
\caption{The geometric construction from the proof of Proposition \ref{prop_313}}
\end{figure}

\medskip Since $B_v, B^v \in \partial_v M$ we know that $B^v(q) = B_v(q)$ and hence $f_v(q) = 0 < \eps$.
Lemma \ref{lem_101} states that there exists $ t_{+} > t_0$ and a geodesic segment $\tau_+$ that connects the point $q$ with the point $\gamma_{p, v}(t_{+})$, such that
\begin{equation}
f_v(y) < \eps \qquad \qquad \qquad \text{for all} \ y \in \tau_+.
\label{eq_642} \end{equation}
Similarly, by Lemma \ref{lem_101} there exists $t_{-} < t_0$ and a geodesic segment $\tau_-$ that connects $q$ and $\gamma_{p,v}(t_-) = \gamma_{p, -v}(-t_{-})$  such that
\begin{equation}
f_v(y) = f_{-v}(y)< \eps \qquad \qquad \qquad \text{for all} \ y \in \tau_-.
\label{eq_1043} \end{equation}
To avoid ambiguity, we stipulate that the geodesic segment $\tau_{\pm}$ contains its endpoints $q$ and $\gamma_{p,v}(t_{\pm})$.
We use the Jordan-Sch\"onflies curve theorem and form a geodesic triangle $T \subset M$, a bounded open
set whose boundary $\partial T$ consists of the three edges  $\tau_-, \tau_+$ and $\gamma_{p,v}([t_-, t_+])$.
The point $q$ is a vertex of this triangle.
It follows from Corollary \ref{cor_131} that any interior point of the geodesic segment from $q$ to a point in
the edge $\gamma_{p,v}((t_-, t_+))$,
belongs to $T$.

\medskip Since $t_- < t_0 < t_+$,
the point $\gamma_{p,v}(t_0)$ is located between the points $\gamma_{p,v}(t_-)$ and $\gamma_{p,v}(t_+)$
along the curve $\gamma_{p,v}$.
Hence $\gamma_{p,v}(t_0)$ is an interior point of the edge of the triangle $T$ that is opposite the vertex $q$.
Since $x$ is an interior point of the geodesic from $q$ to $\gamma_{p,v}(t_0)$,
we conclude  that $x \in T$.
By (\ref{eq_642}) and (\ref{eq_1043}), we know that
\begin{equation}
f_v(y) < \eps \qquad \qquad \qquad \text{for all} \ y \in \tau_- \cup \tau_+.
\label{eq_1043_} \end{equation}
Since $B^v = -B_{-v}$ foliates, there exists a transport line $\eta$ of $B^v$ with $\eta(0) = x \in T$.
Since $\eta(t)$ tends to infinity as $t \rightarrow \infty$, there exists $t > 0$ such that
\begin{equation} \eta(t) \in \partial T. \label{eq_920_} \end{equation}
However, $\gamma_{p,v}$ is a transport line of $B^v \in \partial_v M$, by Proposition \ref{prop_1155}(iii), and hence
$\eta$ and $\gamma_{p,v}$ are disjoint transport lines of $B^v$, by Lemma \ref{lem_1005}.  We thus conclude from (\ref{eq_920_}) that
\begin{equation}
\eta(t) \in \partial T \setminus \gamma_{p,v} \subseteq \tau_- \cup \tau_+. \label{eq_920__} \end{equation}
Since $B_v$ is $1$-Lipschitz, by (\ref{eq_1043_}) and (\ref{eq_920__}),
$$ f_v(x) = B^v(\eta(0)) - B_v(\eta(0)) = B^v(\eta(t)) - t - B_v(\eta(0)) \leq B^v(\eta(t)) - B_v(\eta(t)) = f_v(\eta(t)) < \eps. $$
This completes the proof
that $B^v \equiv B_v$ in $M$. Since $B^v$ is the maximal element of $\partial_v M$ while $B_v$ is the minimal element,
$\partial_v M$ is a singleton.
\end{proof}

\section{No conjugate points}
\label{sec_no_conj}

Proposition \ref{prop_313} will be used in order to show that $M$ has no conjugate points.
First we need:

\begin{lemma} For any $x \in M$, the map  $S^1 \ni v \mapsto \nabla B_v(x) \in S_x M$ is continuous and onto.
 \label{lem_155}
\end{lemma}

\begin{proof} First we prove that the map $v \mapsto B_v(x)$ is continuous in $v \in S^1$, for any fixed $x \in M$.
Recall that the $1$-Lipschitz function $B_v$ vanishes at $0$ for any $v \in S^1$.
Let $v_m \longrightarrow v$ be a sequence in $S^1$. By the Arzela-Ascoli theorem, we may pass to a subsequence and assume that $(B_{v_m} )_{m \geq 1}$ converges locally-uniformly to some $1$-Lipschitz function $B$, and our goal is to prove that $B \equiv B_v$.
In view of Proposition \ref{prop_313}, it suffices
to prove that $B \in \partial_v M$. For any $p\in L$, we have
$$
B(p)=\lim_{m \rightarrow \infty} B_{v_m}(p)=\lim_{m \rightarrow \infty}\langle p, v_m\rangle=\langle p, v\rangle,
$$
and hence $B \in \partial_v M$. We have thus proved  that the map
$  v \mapsto B_v(x)$
is continuous in $v \in S^1$, for any fixed $x \in M$.
Recall from Proposition \ref{lem_357} that $B_v$ foliates for any $v \in S^1$.
Lemma \ref{lem_200} now implies the continuity of the map
$$ S^1 \ni v \mapsto \nabla B_v(x) \in S_x M. $$
From (\ref{eq_639}) we know that $-B_v \in \partial_{-v} M$,
hence $-B_v = B_{-v}$ by Proposition \ref{prop_313}. Therefore $v \mapsto \nabla B_v(x)$
is a continuous, odd,  map from $S^1$ to $S_x M$, hence its Brouwer degree is odd
 and the map is onto.
\end{proof}

\begin{corollary} All geodesics in $M$ are minimizing, so there are no conjugate points in $M$.
\label{prop_159}
\end{corollary}

\begin{proof} Let $x \in M$ and let $\gamma$ be a complete geodesic with $\gamma(0) = x$.
By Lemma \ref{lem_155} there exists $v \in S^1$ such that $\dot{\gamma}(0) = \nabla B_v(x)$.
By Proposition \ref{lem_357}, the geodesic
$\gamma$ is a transport line of $B_v$, and hence it is a minimizing geodesic.
Since a complete, minimizing geodesic cannot contain a pair of conjugate points, there
are no conjugate points in $M$. \end{proof}

Given $x \in \RR^2$ and $r > 0$, we write $x = a v$ for $a \geq 0$ and $v \in S^1$ and define  $$ T_r(x) = \gamma_{0,v}(a r). $$
Then $T_r: \RR^2 \rightarrow M$ is a homeomorphism, by Corollary \ref{cor_131} and Lemma \ref{lem_1008}.
For $x, y \in \RR^2$ and for $r > 0$  denote
$$ d_r(x,y) = \frac{d(T_r(x), T_r(y))}{r}. $$
Recall that we already discussed the large-scale geometry of $M$. In fact, Proposition \ref{lem_445} directly implies the following:

\begin{corollary} For any $x, y \in \RR^2$,
$$ \lim_{r \rightarrow \infty} d_r(x,y) = |x - y| $$
and the convergence is locally uniform in $x,y \in \RR^2$.
\label{cor_122}
\end{corollary}

Write $M_r$ for the Riemannian manifold obtained from $M$ by multiplying the metric tensor by a factor of $1/r^2$.
Then $(\RR^2, d_r)$ is a metric space isometric to $M_r$ via the map $T_r$, where $M_r$ is a complete, connected Riemannian surface in which all geodesics are minimizing.
Write $\mu_r$ for the area measure on $\RR^2$ corresponding to this isometry. That is, $r^2 \cdot \mu_r$ is the measure on $\RR^2$ obtained by pulling back the Riemannian area measure on $M$ under the homeomorphism $T_r$.

\begin{proof}[Proof of Theorem \ref{thm_1217}] Write $D \subseteq \RR^2$ for the open Euclidean unit disc centered at the origin, and observe that $T_r(D) = D_M(0,r)$, where $D_M(0,r) \subset M$ is the open geodesic ball of radius $r$ centered at $0$.  We claim that
 \begin{equation}  \lim_{r \rightarrow \infty} \mu_r( D ) = \pi. \label{eq_237_} \end{equation}
Indeed, from Corollary \ref{cor_122} we know that $d_r$ tends locally uniformly to the Euclidean metric
in $\RR^2$ as $r \rightarrow \infty$. Moreover, for any $r$, the topology induced by $d_r$ on $\RR^2$ is the standard one,
and the metric space $(\RR^2, d_r)$ is isometric to a complete, connected, $2$-dimensional Riemannian manifold in which all geodesics are minimizing.
Proposition \ref{prop_1110}, stated and proven in the appendix below, thus yields (\ref{eq_237_}).
Consequently, by the definition of $T_r, d_r$ and $\mu_r$,
\begin{equation}  \frac{\area( D_M(0, r) )}{r^2} = \mu_r(D) \xrightarrow{r \rightarrow \infty} \pi. \label{eq_239_} \end{equation}
Corollary \ref{cor_131} states that $M$ is diffeomorphic to $\RR^2$. According to Corollary \ref{prop_159} there are no conjugate points in $M$.
In view of (\ref{eq_239_}) we may apply Theorem \ref{thm_BE}, due to Bangert and Emmerich, and conclude that $M$ is flat.
The only flat surface in which all geodesics are minimizing is the Euclidean plane $\RR^2$.
\end{proof}

\begin{proof}[Proof of Corollary \ref{cor_344}] Suppose by contradiction that $X$ embeds isometrically
in a complete, $2$-dimensional, Riemannian manifold $\tilde{M}$. Since all distances in $X$ are finite, we may assume that $\tilde{M}$
is connected. Since $X$ contains a net in a two-dimensional affine plane,
$\tilde{M}$ is necessarily isometric to the Euclidean plane $\RR^2$ by Theorem \ref{thm_1217}. Hence for any four points $x_1,x_2,x_3,x_4 \in \tilde{M}$,
abbreviating $d_{ij} = d(x_i, x_j)$, the $4 \times 4$ matrix
\begin{equation}  \left( \frac{-d_{ij}^2 + d_{i4}^2 + d_{j4}^2}{2} \right)_{i,j=1,\ldots,4} \label{eq_348}
\end{equation}
is of rank at most two. Indeed, the matrix in (\ref{eq_348}) is the Gram matrix of four points in a Euclidean plane. However,
since $X \subseteq \RR^3$ is not contained in a two-dimensional affine plane, there exist four points $x_1,x_2,x_3,x_4 \in X$
whose affine span in $\RR^3$ is three-dimensional. For these four points, the matrix in (\ref{eq_348}) has rank $3$, in contradiction.
\end{proof}

\section{Appendix: Continuity of area}

Suppose that for any $m \geq 1$ we are given a metric $d_m$ on $\RR^2$, such that the following hold:
\begin{enumerate}
\item[(i)] For any $m$, the metric $d_m$ induces the standard topology on $\RR^2$.
\item[(ii)] For any $x,y \in \RR^2$ we have $d_m(x,y) \longrightarrow |x-y|$ as $m \to \infty$, and the convergence is locally uniform.
\item[(iii)] For any $m$, the metric space $(\RR^2, d_m)$ is isometric to a complete, connected, $2$-dimensional Riemannian manifold in which all geodesics are minimizing.
\end{enumerate}

Write $\area_m$ for the Riemannian area measure on $\RR^2$ that corresponds to $d_m$ under the above isometry.

\begin{proposition} For $D = D(0,1) =\{ x \in \RR^2 \, ; \, |x| < 1 \}$ we have $\area_m(D) \longrightarrow \pi$ as $m \rightarrow \infty$.
\label{prop_1110}
\end{proposition}

We were unable to find a proof of Proposition \ref{prop_1110} in the literature,
even though Ivanov's paper \cite{Iv} contains a deeper result which ``almost'' implies this proposition.
A proof of Proposition \ref{prop_1110} is thus provided in this Appendix.

\medskip
By a $d_m$-geodesic in $\RR^2$ we mean a geodesic with respect to the metric $d_m$. The $d_m$-length
of a $d_m$-rectifiable curve $\gamma$ is denoted by $\length_m(\gamma)$.
A set $K \subseteq \RR^2$ is $d_m$-convex if the intersection of any $d_m$-geodesic with $K$ is connected.
All $d_m$-geodesics are minimizing,
and each complete $d_m$-geodesic divides $\RR^2$ into two connected components.
Each of these connected components is a $d_m$-convex, open set called a $d_m$-half-plane.
The intersection of finitely many $d_m$-half-planes, if bounded and non-empty, is called a $d_m$-polygon. Note that our polygons are always open and convex.
The boundary of any $d_m$-polygon consists of finitely many vertices and the same number of edges, and each edge is a $d_m$-geodesic segment.

\medskip
Write $\cG_m$ for the collection
of all complete $d_m$-geodesics in $\RR^2$,
where we identify between two geodesics if they differ by an orientation-preserving reparametrization.
Write $\sigma_m$ for the  Liouville (or \'etendue) measure on $\cG_m$, see Kloeckner and Kuperberg \cite[Section 5.2]{KK} and \'Alvarez-Paiva
and Berck \cite[Section 5]{AB}) and references therein for the basic properties of this measure, and for the formulae of Santal\'o and Crofton
from integral geometry.
 The Santal\'o formula implies that for any open set $U \subseteq \RR^2$,
$$  \area_m( U ) = \frac{1}{2 \pi} \int_{\cG_m} \length_m(\gamma \cap U ) d \sigma_m(\gamma). $$
(We remark that $\gamma \cap U$ can be disconnected, yet it is a disjoint, countable union of $d_m$-geodesics, and $\length_m(\gamma \cap U)$
is the sum of the $d_m$-lengths of these $d_m$-geodesics).
The Crofton formula implies that
 for any $d_m$-polygon $P \subseteq \RR^2$,
$$  \perimeter_m( P ) = \frac{1}{2} \cdot
\sigma_m \left(  \left \{ \gamma \in \cG_m \, ; \, \gamma \cap P \neq \emptyset \right \} \right), $$
where $\perimeter_m(P) = \length_m(\partial P)$.
When we discuss $\area, \length, \perimeter$ or polygons without the subscript $m$ we refer to the usual Euclidean geometry in $\RR^2$.
Write $\cG$ for the collection of all lines in $\RR^2$, where we identify between two lines if they differ by an orientation-preserving reparametrization. Write $\sigma$ for the Euclidean Liouville  measure on $\cG$. We require the following
Euclidean lemma:

\begin{lemma} Let $\tilde{\sigma}$ be a Borel measure on $\cG$ such that for any convex polygon $P \subseteq D(0,2) \subseteq \RR^2$,
\begin{equation}  \perimeter( P ) = \frac{1}{2} \cdot
\tilde{\sigma} \left(  \left \{ \ell \in \cG \, ; \, \ell \cap P \neq \emptyset \right \} \right). \label{eq_126_} \end{equation}
Then,
\begin{equation}  \frac{1}{2 \pi} \int_{\cG} \length(\ell \cap D ) d \tilde{\sigma}(\ell) = \area(D) = \pi. \label{eq_138_} \end{equation}
\label{lem_954}
\end{lemma}

\begin{proof} For any rotation $U \in SO(2)$, the perimeter of the rotated polygon
$U(P) \subset \RR^2$ is the same as the perimeter of $P$. Hence formula (\ref{eq_126_}) holds true with $\tilde{\sigma}$ replaced by $U_* \tilde{\sigma}$,
where by $U_* \tilde{\sigma}$ we mean the push-forward of $\tilde{\sigma}$ under the map $U$ acting on $\cG$ by rotating lines.
Moreover, since $U(D) = D$, replacing $\tilde{\sigma}$ by $U_* \tilde{\sigma}$ does not change the value of the integral on the left-hand side of (\ref{eq_138_}).

\medskip We may thus replace the measure $\tilde{\sigma}$ by the average of $U_* \tilde{\sigma}$ over $U \in SO(2)$, and assume from now on
that $\tilde{\sigma}$ is a rotationally-invariant measure on $\cG$. The validity of (\ref{eq_126_}) for any convex polygon $P \subset D(0,2)$ implies its validity
for all convex sets in the disc $D(0,2)$. Indeed, both the left-hand side and the right-hand side of (\ref{eq_126_}) are monotone in the convex set $P$ under inclusion, and convex polygons are dense in the class of all convex subsets of $D(0,2)$. Consequently, for any $0 < \rho < 2$,
\begin{equation}
2 \pi \rho = \perimeter (D(0,\rho)) = \frac{1}{2} \cdot
\tilde{\sigma} \left(  \left \{ \ell \in \cG \, ; \, \ell \cap D(0,\rho) \neq \emptyset \right \} \right). \label{eq_947_}
\end{equation}
For $\ell \in \cG$ write $r(\ell) = \inf_{x \in \ell} |x| \in [0, \infty)$.
We may reformulate (\ref{eq_947_}) as follows:
\begin{equation}
\tilde{\sigma} \left(  \left \{ \ell \in \cG \, ; \, r(\ell) < \rho \right \} \right) =
\sigma \left(  \left \{ \ell \in \cG \, ; \, r(\ell) < \rho \right \} \right)
\qquad \text {for any } 0 < \rho < 2.
\label{eq_1926}
\end{equation}
Since both $\sigma$ and $\tilde{\sigma}$ are rotationally-invariant measures on $\cG$, they are completely determined by their push-forward under the map $\ell \mapsto r(\ell)$.
From (\ref{eq_1926}) we learn that $\tilde{\sigma}$ coincides with $\sigma$
on the set $\cG \cap r^{-1}([0,2))$.
By the Satanl\'o formula for $\sigma$,
\begin{equation*} \frac{1}{2\pi} \int_{\cG} \length(\ell \cap D) d \tilde{\sigma}(\ell) =\frac{1}{2\pi} \int_{\cG} \length(\ell \cap D) d \sigma(\ell) =
\area(D) = \pi. \end{equation*} \end{proof}

When we refer to the Hausdorff metric below, we always mean the Euclidean Hausdorff metric  (the Hausdorff metric is defined, e.g., in \cite{bbi}).
Write $(x,y) \subseteq \RR^2$ for the Euclidean interval between $x,y \in \RR^2$
excluding the endpoints, and $[x,y] = (x,y) \cup \{ x,y \}$.
We similarly write $[x,y]_m$ and $(x,y)_m$ for the $d_m$-geodesic between $x$ and $y$, with and without the endpoints.
We claim that for any $x,y \in \RR^2$,
\begin{equation} [x,y]_m \xrightarrow{m \to \infty} [x,y] \label{eq_1125} \end{equation}
in the Hausdorff metric. Indeed, for any $0 \leq \lambda \leq 1$, the point on $[x,y]_m$ whose $d_m$-distance
from $x$ equals $\lambda \cdot d_m(x,y)$ must converge to the point on $[x,y]$ whose Euclidean distance
from $x$ equals $\lambda \cdot |x-y|$. It follows from our assumptions that the convergence is uniform
in $\lambda$, and (\ref{eq_1125}) follows.
Moreover, it follows that the Hausdorff convergence in (\ref{eq_1125}) is locally uniform in $x,y \in \RR^2$.

\medskip Write $\overline{A}$ for the closure of a set $A$.
The Euclidean $\eps$-neighborhood of a subset $A \subseteq \RR^2$
is the collection of all $x \in \RR^2$ with $d(x,A) < \eps$ where $d(x,A) = \inf_{y \in A} |x-y|$.
Given a convex polygon $P \subseteq \RR^2$, for a sufficiently large $m$ we  define $P^{(m)} \subseteq \RR^2$ to be the $d_m$-polygon with the same vertices as $P$. We need $m$ to be sufficiently large in order to guarantee that no vertex of $P$ is in the $d_m$-convex hull of the other vertices.

\begin{lemma} Let $P_0, P_1 \subseteq D(0,3)$ be convex polygons such that $\overline{P_0} \subseteq P_1$. Then there exist $m_0 \geq 1$ and $\eps > 0$ such that
the following holds: For any $m \geq m_0$ and any $x, x',y,y' \in D(0,3)$ with $|x -x'| < \eps$ and $|y - y'| < \eps$,
$$
(x,y) \cap P_0 \neq \emptyset \qquad \Longrightarrow \qquad
 (x',y')_m \cap P_1^{(m)} \neq \emptyset,
$$
and
$$
 (x',y')_m \cap P_0^{(m)} \neq \emptyset
\qquad \Longrightarrow \qquad (x,y) \cap P_1 \neq \emptyset.
$$ \label{lem_558}
\end{lemma}

\begin{proof} From the Hausdorff convergence in (\ref{eq_1125}) it follows that for a sufficiently large $m$, the closure of $P_0^{(m)} \cup P_0$ is contained in $P_1 \cap P_1^{(m)}$.
In fact, there exist $\delta > 0$ and $m_1 \geq 1$  such that for $m \geq m_1$, the Euclidean $\delta$-neighborhood
of $P_0^{(m)} \cup P_0$ is contained in $P_1 \cap P_1^{(m)}$.

\medskip Set $\eps = \delta / 2$.
Since the convergence in (\ref{eq_1125}) is uniform in $x,y \in D(0,3)$, there exists $m_0 \geq m_1$  such that for any $m \geq m_0$ and $x',y'\in D(0,3)$, the Hausdorff distance between $(x',y')$ and $(x',y')_{m}$ is at most $\eps$.
Thus for any $m \geq m_0$ and $x, x',y,y' \in D(0,3)$ with $|x -x'| < \eps$ and $|y - y'| < \eps$, the Hausdorff distance between $(x,y)$ and $(x',y')$ is at most $\eps$, and by the triangle inequality, the Hausdorff distance between $(x,y)$ and $(x',y')_{m}$ is at most $2\epsilon = \delta$.

\medskip Hence if $(x,y)$ intersects $P_0$, then $(x',y')_m$ intersects the Euclidean $\delta$-neighborhood of $P_0$, which is contained in $P_1^{(m)}$. Similarly, if $(x',y')_m$ intersects $P_0^{(m)}$,
then $(x,y)$ intersects the $\delta$-neighborhood of $P_0^{(m)}$ which is contained in $P_1$.
\end{proof}

\begin{lemma} Let $K \subseteq \RR^2$ be a bounded, open, convex set. Then there exist $d_m$-polygons $K_m^{\pm} \subseteq \RR^2$ for $m \geq 1$,
 real numbers $\eps_m \searrow 0$ and $m_0 \geq 1$, such that for any  $m \geq m_0$  the following hold:
$$ K_m^- \subseteq K \subseteq K_m^+, $$
 both boundaries $\partial K_m^{\pm}$
are $\eps_m$-close to $\partial K$ in the  Hausdorff metric, and
the $d_m$-perimeters of $K_m^{\pm}$ differ from $\perimeter(K)$ by at most $\eps_m$.
\label{lem_430}
\end{lemma}

\begin{proof} It suffices to show that for any fixed $\eps > 0$ there exist $m_0 \geq 1$ and $d_m$-polygons $K_m^\pm \subset \RR^2$, defined for any $m \geq m_0$, such that 
$$ K_m^- \subseteq K \subseteq K_m^+, $$
and both boundaries $\partial K_m^{\pm}$
are $\eps$-close to $\partial K$ in the Hausdorff metric, and the $d_m$-perimeters of $K_m^{\pm}$ differ from $\perimeter(K)$ by at most $\eps$.

\medskip
Fix $\eps > 0$. We may pick finitely many points in $\partial K$, cyclically ordered, such that when connecting
each point via a segment to its two adjacent points, the result is a convex polygon whose boundary is $(\eps/2)$-close
to $\partial K$ in the Hausdorff metric. We may also require that the perimeter of this convex polygon differs from $\perimeter(K)$ by at most $\eps/2$.

\medskip We slightly move these finitely many points inside $K$, and replace the segments between the points
by $d_m$-geodesics. For a sufficiently large $m$, this defines a $d_m$-polygon $K_m^-$.
It follows from (\ref{eq_1125}) that for a sufficiently large
$m$, the  boundary $\partial K_m^-$ is $\eps$-close to $\partial K$ in the  Hausdorff metric,
the $d_m$-perimeter of $K_m^-$ differs from $\perimeter(K)$ by at most $\eps$,
and $K_m^- \subseteq K$.

\medskip We still need to construct $K_m^+$. Approximate $K$ by a convex polygon containing the closure of $K$ in its interior,
whose boundary is
$(\eps/2)$-close to $\partial K$ in the Hausdorff metric, and whose perimeter differs from $\perimeter(K)$ by at most $\eps/2$.
Replace the edges of this polygon by $d_m$-geodesics in order to form $K_m^+$. It follows from (\ref{eq_1125}) that for
a sufficiently large $m$,
the $d_m$-convex set $K_m^+$ has the desired properties.
\end{proof}

We apply Lemma \ref{lem_430} for the unit disc $D \subseteq \RR^2$, and obtain two $d_m$-polygons $D_m^{\pm}$ with
$D_m^- \subseteq D \subseteq D_m^+$ that satisfy the conclusions of the lemma.
It follows from (\ref{eq_1125}) that for $x,y \in \RR^2$,
\begin{equation}
\length_m( (x,y)_m \cap D_m^{\pm} ) \xrightarrow{m \to \infty} \length( (x,y) \cap D ), \label{eq_1000_}
\end{equation}
and the convergence is locally uniform in $x, y \in \RR^2$.
 Let us fix a convex polygon $K \subseteq \RR^2$ such that
$\overline{D(0,2)} \subseteq K$ and $\overline{K} \subseteq D(0,3)$.
We apply Lemma \ref{lem_430} and obtain $d_m$-polygons $K_m = K_m^+$ for $m \geq 1$ that approximate $K$.
For a set $A \subseteq \RR^2$ denote
$$
\cG(A) = \{ \ell \in \cG \, ; \, \ell \cap A \neq \emptyset \}
\qquad \text{and} \qquad
\cG_{m}(A) = \{ \gamma \in \cG_{m} \, ; \, \gamma \cap A \neq \emptyset \}.
$$

\begin{definition} Define the  map $T_K :\cG(K) \to \partial K \times \partial K \subset \RR^2 \times \RR^2$ by
$$ T_K(\ell) = (a(\ell), b(\ell)) \in \partial K \times \partial K, $$
where $\ell \cap \partial K = \{ a(\ell), b(\ell) \}$ and the line $\ell$ is oriented from the point $a(\ell)$ towards the point $b(\ell)$.
 We analogously define the map $T_{m}: \cG_m(K_m) \rightarrow \partial K_m \times \partial K_m \subset \RR^2 \times \RR^2$ via
$$ T_m(\gamma) = (a(\gamma), b(\gamma)) \in \partial K_m \times \partial K_m, $$
where $\gamma \cap \partial K_m = \{ a(\gamma), b(\gamma) \}$ and the geodesic $\gamma$ is oriented from $a(\gamma)$ towards $b(\gamma)$.
\end{definition}

Denote by $\mu$  the push-forward of $\sigma|_{\cG(K)}$ under the map $T_{K}$,
and denote by $\mu_m^{}$  the push-forward of $\sigma_m|_{\cG(K_m)}$ under the map $T_{m}$.
By Lemma \ref{lem_430} and the Crofton formula,
\begin{equation}
\frac{1}{2} \cdot \mu_m^{}(\RR^2 \times \RR^2) = \perimeter_m(K_m) \xrightarrow{m \rightarrow \infty} \perimeter(K)
= \frac{1}{2} \cdot \mu(\RR^2 \times \RR^2).
\label{eq_1017} \end{equation}
For a convex polygon $P \subseteq \RR^2$ we write $\cF(P) \subseteq \partial K \times \partial K$ for the collection of all pairs of points $x \neq y \in \partial K$ with $(x,y) \cap P \neq \emptyset$.
For a $d_m$-polygon $P$ we denote by $\cF_m(P) \subseteq \partial K_m \times \partial K_m$ the collection of all pairs of points $x \neq y \in \partial K_{m}$ with $(x,y)_m \cap P \neq \emptyset$. Note that if $\overline{P} \subseteq K_m$ then by the Crofton formula,
\begin{equation} \frac{1}{2} \cdot \mu_m^{}(\cF_m(P)) =
\perimeter_m(P). \label{eq_542_} \end{equation}
For a subset $A \subseteq \RR^2 \times \RR^2$ and $\eps > 0$ we write $\cN_{\eps}(A) \subseteq \RR^2 \times \RR^2$ for the Euclidean $\eps$-neighborhood, i.e.,
the collection of all $(x,y) \in \RR^2 \times \RR^2$ for which there exists $(z,w) \in A$ with $|x - z| < \eps$ and $|y - w| < \eps$.

\begin{lemma}  Fix two convex polygons $P_0, P_1 \subseteq \RR^2$ with $  \overline{P_0} \subseteq P_1$ and $\overline{P_1} \subseteq K$.
For $i=0,1$ abbreviate $\cF_i = \cF(P_i)$.
Then there exists $\eps_0 > 0$ such that
\begin{equation}   \limsup_{m \rightarrow \infty} \mu_m^{}(\cN_{\eps_0}(\cF_0)) \leq \mu(\cF_1) = 2 \cdot \perimeter(P_1).
\label{eq_1204} \end{equation}
Furthermore, for any $0 < \eps < \eps_0$,
\begin{equation} \liminf_{m \rightarrow \infty} \mu_m^{}(\cN_{\eps}(\cF_1)) \geq \mu(\cF_0) = 2 \cdot \perimeter(P_0). \label{eq_1205} \end{equation}
\label{lem_1004}
\end{lemma}

\begin{proof}
Recall that $P^{(m)} \subseteq \RR^2$ was defined to be the $d_m$-polygon with the same vertices as $P$, which is well-defined for a sufficiently large $m$. Write $\cF_i^{(m)} = \cF_m(P_i^{(m)})$.
According to  Lemma \ref{lem_558} there exist $\eps_0 > 0$ and $m_0 \geq 1$ such that for any $m \geq m_0$,
\begin{equation}
\cN_{\eps_0}(\cF_0) \cap (\partial K_m \times \partial K_m) \subseteq \cF_1^{(m)}
\qquad \text {and} \qquad
\cN_{\eps_0}(\cF_{0}^{(m)}) \cap (\partial K \times \partial K) \subseteq \cF_1.
\label{eq_2350}
\end{equation}
By increasing $m_0$ if necessary, we may assume that $\overline{P_{i}^{(m)}} \subseteq K \subseteq K_{m}$ for all $m \geq m_0$ and $i=0,1$.
Using (\ref{eq_2350}) and (\ref{eq_542_}), for $m \geq m_0$,
$$ \mu_m^{}( \cN_{\eps_0}(\cF_0) ) \leq \mu_m^{}( \cF_1^{(m)} ) = 2 \cdot \perimeter_m(P_1^{(m)}) \xrightarrow{m \to \infty} 2 \cdot \perimeter(P_1). $$
Fix $0 < \eps < \eps_0$. By Lemma \ref{lem_430},  there exists $m_1 \geq m_0$ such that for any $m \geq m_1$, the  Hausdorff distance between $\partial K_m$ and $\partial K$ is at most $\eps$. From (\ref{eq_2350}) we obtain that for $m \geq m_1$,
$$
\cF_0^{(m)} \subseteq \cN_{\eps}(\cF_1).
$$
Hence,
$$ \mu_m^{}( \cN_{\eps}(\cF_1) ) \geq \mu_m^{}( \cF_0^{(m)} ) = 2 \cdot \perimeter_m(P_0^{(m)}) \xrightarrow{m \to \infty} 2 \cdot \perimeter(P_0). $$
This completes the proof of (\ref{eq_1204}) and (\ref{eq_1205}).
\end{proof}

\begin{proof}[Proof of Proposition \ref{prop_1110}]
By passing to a subsequence, we may assume that $\area_m(D)$ converges
to an element of $\RR \cup \{ + \infty \}$ as $m \to \infty$, and our goal is to prove that this limit equals $\pi = \area(D)$.

\medskip
The total mass of the measures $\mu_m^{}$ is uniformly bounded, by (\ref{eq_1017}). Lemma \ref{lem_430} implies
that the support of $\mu_m^{}$, which is contained in $\partial K_m \times \partial K_m$, is uniformly bounded in $\RR^2$. We may thus pass to a subsequence and assume
that
\begin{equation}  \mu_m^{} \xrightarrow{m \to \infty} \tilde{\mu} \label{eq_957_} \end{equation}
weakly for some measure $\tilde{\mu}$. This means that for any continuous test function $\vphi$ on $\RR^2 \times \RR^2$
we have $\int \vphi d \mu_m^{} \longrightarrow \int \vphi d \tilde{\mu}$.
The measure $\tilde{\mu}$ is supported on $\partial K \times \partial K$, by Lemma \ref{lem_430}.

\medskip Recall that $\overline{D(0,2)} \subseteq K$. We claim  that for any convex polygon $P \subseteq \RR^2$ with $P \subseteq D(0,2)$,
\begin{equation}
\tilde{\mu}(\cF(P)) = \mu(\cF(P)).
\label{eq_952} \end{equation}
Since $\mu(\cF(P)) = 2 \cdot \perimeter(P)$ is continuous in $P$ and monotone in  $P$ under inclusion, and since $\cF(P') \subseteq \cF(P)$ when $P' \subseteq P$,
in order to prove (\ref{eq_952}) it suffices to prove the following: For any two convex polygons $P_0, P_1 \subseteq \RR^2$ with $ \overline{P_0} \subseteq P_1$ and $\overline{P_1} \subseteq K$,
\begin{equation}  \tilde{\mu}(\cF(P_0)) \leq \mu(\cF(P_1)) \qquad \text{and} \qquad \mu(\cF(P_0)) \leq \tilde{\mu}(\overline{\cF(P_1)}). \label{eq_1004}
\end{equation}
From Lemma \ref{lem_1004} and (\ref{eq_957_}), there exists $\eps_0 > 0$ with
$$ \tilde{\mu}(\cF(P_0)) \leq \limsup_{m \rightarrow \infty}  \mu_m(\cN_{\eps_0}(\cF(P_0)))
\leq \mu(\cF(P_1)) $$
and for any $0 < \eps < \eps_0$,
$$ \tilde{\mu}(\cN_{2 \eps}(\cF(P_1))) \geq \liminf_{m \rightarrow \infty} \mu_m( \cN_{\eps}(\cF(P_1)) ) \geq \mu(\cF(P_0)).
$$
By letting $\eps$ tend to zero, we obtain (\ref{eq_1004}), and hence (\ref{eq_952}) is proven.
The map $T_K^{-1}$ is a well-defined map from $A = \{ (x,y) \in \partial K \times \partial K \, ; \, x \neq y \}$ to $\cG$.
By (\ref{eq_952}), the push-forward of $\tilde{\mu}|_A$ under the map $T_K^{-1}$ is a measure $\tilde{\sigma}$ on $\cG$
which satisfies the assumptions of Lemma \ref{lem_954}. From the conclusion of Lemma \ref{lem_954},
\begin{equation}  \frac{1}{2 \pi} \int_{\partial K \times \partial K} \length((x,y) \cap D ) d \tilde{\mu}(x,y) = \frac{1}{2 \pi} \int_{A} \length((x,y) \cap D ) d \tilde{\mu}(x,y) = \pi. \label{eq_955_} \end{equation}
By the Santal\'o formula,
$$ \area_m(D_m^{\pm}) =
\frac{1}{2\pi} \int_{\RR^2 \times \RR^2} \length_m( (x,y)_m \cap D_m^{\pm} ) d \mu_m^{}(x,y). $$
We thus deduce from (\ref{eq_1000_}),  (\ref{eq_957_}) and (\ref{eq_955_}) that
\begin{equation}  \lim_{m \rightarrow \infty} \area_m(D_m^{\pm}) = \frac{1}{2\pi} \int_{\RR^2 \times \RR^2} \length( (x,y) \cap D) d \tilde{\mu}(x,y) = \pi. \label{eq_1207} \end{equation}
However, $D_m^- \subseteq D \subseteq D_m^+$. Hence (\ref{eq_1207}) implies that $\area_m(D) \longrightarrow \pi$.
\end{proof}

\medskip
\noindent Department of Mathematics, Weizmann Institute of Science, Rehovot 76100, Israel. \\
 {\it e-mails:} \verb"matan.eilat@weizmann.ac.il, boaz.klartag@weizmann.ac.il"

\end{document}